\documentclass[a4paper,11pt]{amsart}
\usepackage{rotating}
\usepackage{pdflscape}
\usepackage{amsmath}
\usepackage{amssymb}
\usepackage{MnSymbol}
\usepackage{color}
\usepackage{enumerate}
\usepackage{amsthm}
\usepackage{listings}
\usepackage[all]{xy}
\usepackage{longtable}
\usepackage{graphicx}
\usepackage{hyperref}
\usepackage{mathrsfs}
\usepackage[style=alphabetic,backend=bibtex]{biblatex}
\theoremstyle{definition}
\newtheorem{theorem}{Theorem}[section]
\newtheorem{proposition}[theorem]{Proposition}
\newtheorem{lemma}[theorem]{Lemma}
\newtheorem{corollary}[theorem]{Corollary}
\newtheorem{definition}[theorem]{Definition}
\newtheorem{remark}[theorem]{Remark}
\newtheorem{example}[theorem]{Example}

\newcommand{\reg}{\mathrm{reg}}

\newcommand{\Kosz}{\operatorname{Kosz}}

\newcommand{\Hom}{\operatorname{Hom}}

\newcommand{\rank}{\operatorname{rank}}

\newcommand{\D}{\operatorname{d}\!}
\newcommand{\CC}{\mathbb{C}}
\newcommand{\QQ}{\mathbb{Q}}
\newcommand{\ZZ}{\mathbb{Z}}

\newcommand{\RR}{\mathbb{R}}

\newcommand{\OO}{\mathcal{O}}

\newcommand{\id}{\operatorname{id}}

\newcommand{\Grass}{\operatorname{Grass}}

\newcommand{\Stief}{\operatorname{Stief}}

\newcommand{\mult}{\mathrm{mult}}

\newcommand{\sym}{\mathrm{sym}}
\newcommand{\Jac}{\operatorname{Jac}}

\title{Some L\^e-Greuel type formulae on stratified spaces}

\author{Matthias Zach}

\begin{document}

\begin{abstract}
  We extend the circle of ideas from \cite{Zach22} 
  to functions $f \colon (\CC^n, 0) \to (\CC^k, 0)$ 
  with an isolated singularity in a stratified sense 
  on an arbitrary, but fixed complex analytic germ 
  $(X, 0)$. An extension of Tib{\u a}r's Bouquet 
  Theorem for hypersurfaces to this setup allows for 
  a topological definition of Milnor numbers $\mu(\alpha; f)$ 
  for each stratum $V^\alpha$ of $X$ and we prove 
  several formulas which compute these numbers 
  as (alternating) sums of certain ``homological indices''.
  The main technical result at work in the background 
  is a local Riemann-Roch type theorem, relating 
  a topological obstruction to holomorphic Euler 
  characteristics.
\end{abstract}

\maketitle

\tableofcontents

\section{Introduction and results}

\subsection{The classical L\^e-Greuel Formula}
A classical formula due to L\^e \cite{Le74} and Greuel \cite{Greuel75} 
allows one to compute the Milnor 
number $\mu(f)$ of an \textbf{I}solated \textbf{C}omplete \textbf{I}ntersection 
\textbf{S}ingularity (ICIS for short) of a holomorphic map germ 
\[
  f \colon (\CC^n, 0) \to (\CC^k, 0).
\]
For such singularities, one has the Milnor-L\^e fibration 
\[
  f \colon B_\varepsilon \cap f^{-1}(D_\delta \setminus \Delta) \to D_\delta \setminus \Delta
\]
for suitable $1\gg \varepsilon \gg \delta > 0$ over the complement of the discriminant 
$(\Delta, 0) \subset (\CC^k, 0)$, the set of critical values; see e.g. \cite{HandbookVol1}
for details. Hamm has shown in \cite{Hamm72} that the fiber 
\[
  M_f = B_\varepsilon \cap f^{-1}(\{v\}), \quad v \in D_\delta \setminus \Delta
\]
is homotopy equivalent to a bouquet of spheres 
\begin{equation}
  \label{eqn:ICISMilnorNumber}
  M_f \cong_{\mathrm{ht}} \bigvee_{j=1}^{\mu(f)} S^{n-k} 
\end{equation}
of real dimension $n-k$ and it is the number of these spheres which we will refer to 
as the Milnor number of $f$. 

Now the formula by L\^e and Greuel suggests that after a sufficiently general 
change of coordinates in the codomain $(\CC^k, 0)$ of $f$ we can compute $\mu(f) = \sum_{j=1}^k (-1)^{j-1} \lambda_j(f)$ as an alternating sum where 
\begin{equation}
  \label{eqn:OriginalLeGreuel}
  \lambda_j(f) = \dim_\CC \left(
  \CC\{x_1,\dots,x_n\} /\left(\left\langle f_1, \dots, f_{k-j} \right\rangle + J_{k-j+1} \right)
  \right)
\end{equation}
and
\[
  J_p = 
  \left\langle
  \det
  \begin{pmatrix}
    \frac{\partial f_1}{\partial x_{i_1}} & \dots & 
    \frac{\partial f_1}{\partial x_{i_p}} \\
    \vdots & \ddots & \vdots \\
    \frac{\partial f_p}{\partial x_{i_1}} & \dots & 
    \frac{\partial f_p}{\partial x_{i_p}} \\
  \end{pmatrix}
  : 0 < i_1 < \dots < i_p \leq n
  \right\rangle 
\]
is the ideal of $p$-minors of the Jacobian matrix of $(f_1, \dots, f_p)$.

\subsection{The central new results}
In this article we describe several ways to generalize these findings to 
holomorphic functions $f \colon (\CC^n, 0) \to (\CC^k, 0)$ which have an 
isolated singularity in a stratified sense 
(cf. Definition \ref{def:IsolatedStratifiedCompleteInteresectionSingularity}) 
on an \textit{arbitrary} reduced, complex analytic germ $(X, 0) \subset (\CC^n, 0)$ in 
its domain. 

Namely, if $\{V^\alpha\}_{\alpha \in A}$ are the strata of 
a fixed Whitney stratification of $X$, then Tib{\u a}r's Bouquet Theorem 
\cite{Tibar95} can be extended to 
\[
  M_{f|(X, 0)} \cong \mathcal L^{k-1}(X, \{0\}) \vee 
  \bigvee_{\alpha \in A \setminus \{0\}} \bigvee_{i = 1}^{\mu(\alpha; f)}
  S^{d(\alpha) - k + 1}(\mathcal L(X, V^\alpha))
\]
(see Theorem \ref{thm:BouquetDecompositionsCompleteIntersections} below)
where we denote by $M_{f|(X, 0)}$ the Milnor fiber of the restriction $f|(X, 0)$. 
The numbers $\mu(\alpha; f)$ generalize the Milnor number 
for classical ICIS from Equation (\ref{eqn:ICISMilnorNumber}) in a natural way. 
In Theorem \ref{thm:LeGreuelFormulaI} we provide several formulas for them 
in terms of holomorphic Euler characteristics; 
one of them being, 
for instance, 
\begin{eqnarray*}
  \mu(\alpha; f)
  & = & 
  \sum_{j=1}^k 
  \chi\left(
  R(\nu_\alpha)_*\left(
  \begin{array}{c}
    {\mathcal K}osz(\nu_\alpha^*(f_1,\dots,f_{j-1}, l_{j+1},\dots, l_k))\\
    \otimes \\
    {\mathcal E}NC(\nu_\alpha^*(\D f_1,\dots, \D f_j, \D l_{j+1},\dots, \D l_k))
  \end{array}
  \right)
  \right)\\
  & & - 
  \sum_{j=1}^k 
  \chi\left(
  R(\nu_\alpha)_*\left(
  \begin{array}{c}
    {\mathcal K}osz(\nu_\alpha^*(f_1,\dots,f_{j-1}, l_{j+1},\dots, l_k))\\
    \otimes \\
    {\mathcal E}NC(\nu_\alpha^*(\D f_1,\dots, \D f_{j-1}, \D l_{j},\dots, \D l_k))
  \end{array}
  \right)
  \right).
\end{eqnarray*}
They involve the Nash modification $\nu_\alpha \colon \tilde X_\alpha \to X_\alpha$ 
of the closures $X_\alpha = \overline{V^\alpha}$ of the strata and sheafified 
versions of the Koszul and the Eagon-Northcott complex. 
This is a natural continuation and extension of the ideas presented in \cite{Zach22}
and these formulas should be thought of as ``homological indices'' 
in the sense of e.g. \cite{GomezMont98} and \cite{EbelingGuseinZadeSeade04}. 

\subsection{Generic complexes and homological indices}
To briefly sketch the connection with the classical L\^e-Greuel Formula, 
recall that the Milnor number for a hypersurface singularity (i.e. for $k=1$) on $(\CC^n, 0)$ 
can be reinterpreted as a homological index
\[
  \mu(f) = \chi\left(\Kosz\left(\frac{\partial f}{\partial x_1}, \dots, \frac{\partial f}{\partial x_n}
  \right) \right).
\]
Given that $f$ has isolated singularity, the partial derivatives of $f$ must form a regular 
sequence, so that the associated Koszul complex is exact but in degree zero. The homology 
there is the Milnor algebra whose length is known to be equal to the Milnor number 
$\mu(f)$ for hypersurfaces. 

Similar considerations can be made for the quotient modules appearing in 
the more general L\^e-Greuel Formula (\ref{eqn:OriginalLeGreuel}):
\begin{equation}
  \label{eqn:LeGreuelAsEulerChar}
  \lambda_j(f) = 
  \chi\left(
  \Kosz(f_1,\dots,f_{k-j}) \otimes \mathrm{ENC}\left(\Jac(f_1,\dots, f_{k-j+1})\right)
  \right),
\end{equation}
where again $\Kosz$ stands for the Koszul- and $\mathrm{ENC}$ for the Eagon-Northcott complex 
(which we recall below). It therefore comes as no surprise that this combination of 
complexes will also play a key role in the main results of this paper. 

\medskip
The general goal is to find a \textit{generic} complex $C^\bullet$ over some extended ring 
$\CC\{x\}[T]$ in additional, auxillary variables, say $T = T_1,\dots, T_N$, also called 
\textit{parameters}. The terms 
$C^p$ must be independent of $T$ in the sense that they
are of the form $M^p \otimes_{\CC\{x\}} \CC\{x\}[T]$ for some $\CC\{x\}$-modules $M^p$. 
The differentials $\psi \colon C^p \to C^{k+1}$ may involve the parameters $T_j$.

Both the Koszul- and the Eagon-Northcott complexes are examples of such generic complexes: 
Instead of constructing them from a sequence or a matrix of ring elements, 
we can also write them down as complexes over 
corresponding rings of indeterminates. These ideas have been introduced and developed 
for instance in \cite{EagonNorthcott67}. 

For such a generic complex $C^\bullet$ to be useful for a particular application, it 
should, after substitution of the parameters $T_i$, exhibit certain desirable properties. 
In the concrete example of the L\^e-Greuel formula above, this means that we start with 
\[
  \Kosz(c) \otimes \mathrm{ENC}(A)
\]
where $c = (c_1,\dots, c_{k-j})$ are indeterminates and $A$ is another matrix of 
indeterminates $a_{s, t}$ over the ring $\CC\{x\}$.
For any given ICIS $f \colon (\CC^n, 0) \to (\CC^k,0)$ we can then substitute 
according to Equation (\ref{eqn:LeGreuelAsEulerChar}) to compute $\lambda_j(f)$ from it. 

\medskip
In a similar vein, generic complexes for homological indices have been defined 
in \cite{GomezMont98} for vector fields with isolated singularity on an isolated 
hypersurface singularity, in \cite{EbelingGuseinZadeSeade04} for $1$-forms 
with isolated zero on arbitrary complex analytic spaces germs with isolated singularity, 
and also in \cite{GorskyGuseinZade18} for collections of $1$-forms.  

We will show in Corollary \ref{cor:bakingComplexes} that in the context 
of our generalized L\^e-Greuel Formula in Theorem \ref{thm:LeGreuelFormulaI}, 
for every stratum $V^\alpha$ of $(X, 0) \subset (\CC^n, 0)$ 
there exists such a generic complex $\mathrm{Prep}(X_\alpha; k, m)$ depending only 
on the closure of the stratum $X_\alpha = \overline {V^\alpha}$ and such that after 
appropriate substitution, its Euler characteristic computes the summands for $\mu(\alpha; f)$.

\medskip

The main workhorse behind our L\^e-Greuel formulas in Theorem \ref{thm:LeGreuelFormulaI}
is a technical result which we believe might enjoy wider applicability in the circle 
of ideas presented here; namely Theorem \ref{thm:MainWorkhorse}.
It is a local Riemann-Roch type theorem in the loose sense that it
relates a topological obstruction on the left hand side with a holomorphic Euler 
characteristic on the right hand side. In particular, it shows how the summands of 
the $\mu(\alpha; f)$ can actually be interpreted as localized Chern classes.

\section{Topological preliminaries}

\subsection{A brief review of the Handlebody Theorem}
Using the ``Carrousel'' developed by L\^e and his own ``Handlebody Theorem'', 
Mihai Tib{\u a}r has established the following ``Bouquet Theorem'' in \cite{Tibar95}:

\begin{theorem}
  \label{thm:BouquetDecompositionHypersurface}
  Let $(X, 0) \subset (\CC^n, 0)$ be a reduced complex analytic space of 
  dimension $\dim (X, 0) \geq 2$\footnote{It is not explicitly stated but seemingly understood from the proof in \cite{Tibar95} that actually no component of $(X, 0)$ must have dimension $<2$}, 
  endowed 
  with a Whitney stratification $X = \bigcup_{\alpha \in A} V^\alpha$ with 
  connected strata of dimensions $d(\alpha) := \dim_\CC V^\alpha$. Suppose 
  \[
    f \colon (\CC^n, 0) \to (\CC, 0)
  \]
  is a complex analytic germ such that the restriction $f|_X$ has an isolated 
  singularity in the stratified sense and let 
  \[
    M_{f|(X, 0)} = B_\varepsilon \cap X \cap f^{-1}(\{v\}), \quad 1 \gg \varepsilon \gg |v| > 0
  \]
  be its stratified Milnor fiber. Then 
  \[
    M_{f|(X, 0)} \cong_{\mathrm{ht}} \mathcal L(X, \{0\}) \vee \bigvee_{\alpha \in A} 
    \bigvee_{j=1}^{\mu(\alpha; f)} S^{d(\alpha)}(\mathcal L(X, V^\alpha))
  \]
  with $\mu(\alpha; f) = 0$ whenever $d(\alpha) = 0$. 
\end{theorem}

By $\mathcal L(X, V)$ we denote the \textit{complex link} of the stratified space $X$ 
along a stratum $V$, see e.g. \cite{GoreskyMacPherson88}, and $S^k(Y)$ denotes 
the $k$-fold repeated suspension of a topological space $Y$.

\begin{example}
  \label{exp:RunningExampleIntroduction}
  Consider the space of square symmetric $2\times 2$-matrices of rank $\leq 1$
  \[
    X = \left\{\varphi = \begin{pmatrix} x & z \\ z & y\end{pmatrix} \in \CC^{2\times 2}_{\sym} : \det(\varphi) = 0\right\} 
    \subset  \CC^{2\times 2}_{\sym} \cong \CC^3
  \]
  endowed with the \textit{rank stratification}
  \[
    X = 
    \underbrace{\{\varphi \in \CC^{2\times 2}_{\sym} : \rank \varphi = 1\}}_{=: V^1} 
    \overset{\cdot}{\cup}
    \underbrace{\{\varphi \in \CC^{2\times 2}_{\sym} : \rank \varphi = 0\}}_{=: V^0 = \{0\}} 
  \]
  and the function 
  \[
    f \colon \CC^{2\times 2}_{\sym} \to \CC^1, \quad \begin{pmatrix} x & z \\ z & y\end{pmatrix} 
    \mapsto y - x^k.
  \]
  Note that $f$ is non-singular on $(\CC^{2\times 2}_{\sym}, 0)$, but 
  its restriction to $X$ has an isolated singularity in the stratified sense 
  (cf. Definition \ref{def:IsolatedStratifiedCompleteInteresectionSingularity} below).

  This example has been chosen so that we can easily compare with classical approaches. 
  The central fiber $X \cap f^{-1}(\{0\})$ is isomorphic to the singularity 
  \[
    h \colon (\CC^2, 0) \to (\CC^1, 0), \quad (x, z) \mapsto \det \begin{pmatrix} x & z \\ z & x^k \end{pmatrix} = x^{k+1} - z^2.
  \]
  We know that the Milnor fiber of this singularity is $M_h \cong_{\mathrm{ht}} \bigvee_{i=1}^{k} S^1$,
  a bouquet of $\mu(h) = k$ spheres, where $\mu(h)$ is the classical Milnor number of 
  the isolated hypersurface singularity $h$. 

  On the contrary, we find that 
  $\mathcal L(X, \{0\}) \cong_{\mathrm{ht}} S^1$
  already contributes one sphere to the bouquet decomposition of $M_{f|(X, 0)}$ 
  in Theorem \ref{thm:BouquetDecompositionHypersurface}. 
  Therefore, since clearly $M_h \cong M_{f|(X, 0)}$, we must have $\mu(1; f) = k-1$.
\end{example}

The numbers $\mu(\alpha; f)$ can be computed as follows; we refer to \cite{Tibar95} 
for the details. For a linear 
form $l \colon \CC^n \to \CC$ we can consider the \textit{relative polar variety} 
\[
  \Gamma_\alpha(f, l) = \overline{\{x \in V^\alpha : \D f(x) \textnormal{ and } \D l(x) \textnormal{ are linearly dependent on } T_x V^\alpha\}}.
\]
When $l$ is chosen in sufficiently general position, then this is known to be a curve and, moreover, 
the restriction of the composite map
\[
  \Phi : (X, 0) \to (\CC^2, 0), \quad x \mapsto (f(x), l(x))
\]
to $\Gamma_\alpha(f, l)$ is a \textit{finite} map which is generically $1:1$ onto its image
\[
  \Phi|_{\Gamma_\alpha(f, l)} : (\Gamma_\alpha(f, l), 0) \to (\Delta_\alpha(f, l), 0) \subset (\CC^2, 0).
\]
For a suitably chosen representative of $(X, 0)$ and a polydisc $D \subset \CC^2$
the map $\Phi \colon X \to D$ is a topological fiber bundle away from the points of 
\[
  \Delta := \bigcup_{\alpha \in A} \Delta_\alpha(f, l).
\]
Its fiber $M_{f, l} = \Phi^{-1}(\{(\delta, \eta)\})$ 
is a hyperplane section of the Milnor fiber $M_f$ of $f$.
Up to homeomorphism, we recognize the latter as $M_f \cong \Phi^{-1}(\{\delta\} \times \CC)$.
With the genericity assumptions made, 
the function $l$ can be assumed to have only complex Morse singularities on the interior of $M_f$.
These sit precisely at the points of the intersection $M_f \cap \Gamma_\alpha(f, l)$ 
for the various $\alpha \in A$. It can be shown that indeed we can recover 
$M_f$ from $M_{f, l}$ up to homotopy by attaching a \textit{thimble} 
$S^{d(\alpha)}(\mathcal L(X, V^\alpha))$ to $M_{f, l}$ for every such 
Morse critical point on $V^\alpha$. The number of these points is the 
multiplicity 
\[
  \mathrm{mult}_l(\Gamma_\alpha(f, l),0).
\]
Similarly, the Milnor fiber $M_l$ of $l$ on $(X, 0)$, i.e. the complex link 
$\mathcal L(X, \{0\})$ of $X$ at the origin,
is nothing but $\Phi^{-1}(\CC\times \{\eta\})$. Again, we can recover $M_l$ 
from $M_{f, l}$ by attaching thimbles at Morse critical points, this time for the 
function $f$ on $M_l$. The number of these points is then equal to 
\[
  \mathrm{mult}_f(\Gamma_\alpha(f, l), 0).
\]
The key insight of Tib{\u a}r was that the ``carrousel monodromy'' can be used to
show that, up to homotopy, we can identify $M_l$ with a subspace 
of the Milnor fiber $M_f$ and, moreover, for every \textit{additional} thimble 
$e$ which is needed to complete $M_l$ to $M_f$, there is already one thimble $e'$ 
attached to $M_{f, l}$ in $M_l$ whose boundary $\partial e'$ is homotopy equivalent 
to the boundary $\partial e$ of $e$. This leads to Tib{\u a}r's 
\textit{Handlebody theorem} in \cite{Tibar95}:

\begin{theorem}
  The Milnor fiber $M_f$ is obtained from the complex link $M_l$ by attaching 
  thimbles for stratified Morse singularities. The boundary of each such attaching 
  map retracts to a point in $M_l$.
\end{theorem}

While the topology of the space $M_l \cong \mathcal L(X, 0)$ remains difficult to 
compute in general, the Handlebody theorem and its build-up allow us to compute the 
number of thimbles as the difference 
\begin{equation}
  \label{eqn:ThimbleNumber}
  \mu(\alpha; f) = \mathrm{mult}_l(\Gamma_\alpha(f, l), 0) - \mathrm{mult}_f(\Gamma_\alpha(f, l), 0).
\end{equation}
For what follows we need to note that Tib{\u a}r's results have a natural extension to the 
zero-dimensional case:

\begin{proposition}
  \label{prp:TibarDimensionZero}
  Let $(X, 0) \subset (\CC^n, 0)$ be a reduced analytic curve and $f \colon (X, 0) \to (\CC, 0)$ a 
  holomorphic function which is not constant on any branch of $(X, 0)$.
  Then the Milnor fiber is a collection of points
  \[
    M_f \cong_{\mathrm ht} \left\{q_1, \dots, q_m\right\}
  \]
  the number of which is equal to the multiplicity $m = \mathrm{mult}_f(X, 0)$ 
  of $f$. 
  Similarly, for a sufficiently general linear form $l \colon (\CC^n, 0) \to (\CC, 0)$ 
  we find that 
  \[
    M_l \cong_{\mathrm ht} \left\{p_1, \dots, p_e\right\}
  \]
  with $e = \mult_l(X, 0)$ the multiplicity of the curve $(X, 0)$.
\end{proposition}

Note that $(X, 0)$ agrees with the relative polar curve $\Gamma(f,l)$ 
for the pair $(f, l)$ by 
definition. In particular, we obtain $M_f$ from $M_l$ by ``attaching'' 
\[
  \mathrm{mult}_l(\Gamma(f, l), 0) - \mathrm{mult}_f(\Gamma(f, l), 0)
\]
many points to it.

\subsection{Morsifications and homological decompositions}

Theorem \ref{thm:BouquetDecompositionHypersurface} has an immediate homological 
counterpart:
\begin{equation}
  \label{eqn:HomologicalBouquetDecomposition}
  \tilde H_k(M_f) \cong \tilde H_k\left(\mathcal L(X, \{0\})\right) \oplus
  \bigoplus_{\alpha \in A} \bigoplus_{i = 1}^{\mu(\alpha; f)} 
  H_{k-d(\alpha)+1}\left(C(\mathcal L(X, V^\alpha)), \mathcal L(X, V^\alpha)\right)
\end{equation}
Here we denote by $\tilde H$ the \textit{reduced homology} with integer coefficients. 

A similar homology decomposition can also be proved using Morsifications instead of the carrousel, 
see e.g. \cite[Proposition 1]{Zach22}. In this case one considers a perturbation 
\[
  f_t := f + t\cdot l \colon X \to \CC
\]
on a suitable representative $X$ of $(X, 0)$. Again, for a generic linear form 
$l$ and sufficiently small $|t| > 0$, the function $f_t$ has only stratified Morse singularities 
on the interior of $X$ and, hence, the whole space $X$ itself is obtained from 
the Milnor fiber $M_f$ by attaching a thimble $(C(\mathcal L(X, V^\alpha)), \mathcal L(X, V^\alpha))$ 
for every stratified Morse critical point of $f_t$ on $V^\alpha$. 
We denote the number of such Morse critical points by $\mu'(\alpha; f)$.

It can be shown that the relative homology of the pair $(X, M_f)$ splits as a direct sum 
with contributing summands supported at the critical points. As $X$ itself is contractible, 
we eventually find that 
\begin{equation}
  \label{eqn:HomologyDecompositionMorsification}
  \tilde H_k(M_f) \cong
  \bigoplus_{\alpha \in A} \bigoplus_{i = 1}^{\mu'(\alpha; f)} 
  H_{k-d(\alpha)+1}\left(C(\mathcal L(X, V^\alpha)), \mathcal L(X, V^\alpha)\right).
\end{equation}
Note that, if we suppose that $\{0\} =: V^0$ is a stratum, then $\mu'(0; f) = 1$ 
and the contribution to the homology of $M_f$ from that point is precisely 
$\tilde H_\bullet\left(\mathcal L(X, \{0\})\right)$.

\begin{example}
  \label{exp:RunningExampleMorsification}
  We continue with Example \ref{exp:RunningExampleIntroduction} and choose 
  \[
    l \colon \CC^{2\times 2}_{\sym} \to \CC, \quad 
    \begin{pmatrix}
      x & z \\ z & y
    \end{pmatrix}
    \mapsto x + y,
  \]
  so that the Morsification of $f$ becomes 
  \[
    f_t = f + t \cdot l \colon 
    \begin{pmatrix}
      x & z \\ z & y
    \end{pmatrix}
    \mapsto 
    (1+t) \cdot y - (x^{k-1} - t) \cdot x.
  \]
  A straightforward calculation reveals that for $t \neq 0$ one has 
  $k-1$ Morse critical points of $f_t$ on the stratum $V^1$ at the points 
  \[
    (x, y, z) = \left(\sqrt[k-1]{\frac{t}{k}}, 0, 0\right).
  \]
  This shows $\mu'(1; f) = k-1$ in accordance with the findings 
  in Example \ref{exp:RunningExampleIntroduction}. 
  We will see in Theorem \ref{thm:MorseNumbersAreEqualToThimbleNumbers} below that 
  $\mu'(\alpha; f) = \mu(\alpha; f)$ in general.
\end{example}

The contents of the following theorem have already been established by D. Massey 
in \cite{Massey96}. In \cite[Theorem 3.1]{MaximTibar23}, L. Maxim and M. Tib{\u a}r give a more 
direct proof. In order to be relatively self-contained, we briefly reproduce the argument here. 

\begin{theorem}
  \label{thm:MorseNumbersAreEqualToThimbleNumbers}
  In the setup of Theorem \ref{thm:BouquetDecompositionHypersurface}, let 
  \[
    f_t = f + t\cdot l : (X, 0) \to (\CC, 0)
  \]
  be a Morsification of $f$ for some sufficiently general linear form 
  $l$ and suppose that $\{0\} = V^0 \subset X$ is a stratum.
  Then for $1 \gg |t| > 0$ sufficiently small and every $\alpha \in A$ 
  the number $\mu'(\alpha;f)$ of Morse critical points of $f_t$ on the stratum $V^\alpha$ is 
  equal to the number $\mu(\alpha; f)$ of summands for $\alpha$ in the Bouquet 
  decomposition in Theorem \ref{thm:BouquetDecompositionHypersurface}, or
  the decomposition in Proposition \ref{prp:TibarDimensionZero}, respectively.
\end{theorem}

\begin{proof}
  We proceed by induction on the dimension $\dim(X, 0)$. For dimension $1$ 
  we consider the Milnor fiber 
  \[
    M_f = \{q_1,\dots, q_m\} = B_\varepsilon \cap X \cap f^{-1}(\{\delta\})
  \]
  for a Milnor ball of radius $\varepsilon > 0$ and a successively chosen 
  regular value $\delta \neq 0$. Perturbing $f$ slightly and passing to 
  $f_t$ does not change the topology of $M_f$ for $|t|$ small enough. 
  Hence $M_f \cong B_\varepsilon \cap X \cap f_t^{-1}(\{\delta\})$ and 
  $f_t$ has only stratified Morse singularities on the interior of 
  the curve $B_\varepsilon \cap X$.

  For the singularity at $\{0\} \subset X$, the boundary of the thimble 
  to be attached to $M_f$ is precisely the complex link $\mathcal L(X, \{0\})$, 
  as $\D f(0) = 0$ and $l$ is chosen to be linear and sufficiently generic. 
  We therefore have $\mu'(0; f) = 1$ and a contribution of 
  $\tilde H_0(\mathcal L(X, \{0\}))$ in (\ref{eqn:HomologyDecompositionMorsification})
  for this stratum. 


  All other strata $V^\alpha$ are one-dimensional. The boundary of any thimble 
  to be attached at a critical point of $f_t$ on such a stratum 
  consists of two points, so its reduced homology contributes a 
  free $\ZZ$-module of rank $1$ in homological degree zero. 
  Overall we find that, on the one hand,
  \[
    \ZZ^{m-1} \cong \tilde H_0(M_f) \cong 
    \underbrace{\tilde H_0(\mathcal L(X, \{0\}))}_{\alpha = 0} \oplus
    \bigoplus_{\alpha \in A \setminus\{0\}} \bigoplus_{i = 1}^{\mu'(\alpha; f)} 
    \ZZ^1.
  \]
  due to e.g. \cite[Proposition 1]{Zach22}. 
  On the other hand, the Bouquet decomposition, Proposition \ref{prp:TibarDimensionZero} 
  indicates that 
  \[
    \ZZ^{m-1} \cong \tilde H_0(M_f) \cong 
    \tilde H_0(\mathcal L(X, \{0\})) \oplus
    \bigoplus_{\alpha \in A \setminus\{0\}} \bigoplus_{i = 1}^{\mu(\alpha; f)} 
    \ZZ^1.
  \]
  Repeating the above considerations for $(X, 0) = (\overline{V^\alpha}, 0)$, 
  i.e. replacing $X$ by the closure of any one of its irreducible, one-dimensional strata, 
  it becomes clear that indeed $\mu(\alpha; f) = \mu'(\alpha; f)$ for every $\alpha \neq 0$. 

  Now suppose that the claim has been established for germs $(X', 0) \subset (\CC^n, 0)$ 
  of dimension $d' < d = \dim(X, 0)$ and let $V^\alpha$ be a top-dimensional, irreducible 
  stratum of $X$. Then $f$ has an isolated singularity in the stratified sense on 
  the closure of that stratum $(X^\alpha, 0) = (\overline V^\alpha, 0)$. 
  Let $A_{\leq \alpha}$ be the set of indices $\beta \in A$ such that 
  the stratum $V^\beta \subset \overline{V^\alpha}$ is equal or adjacent to $V^\alpha$.
  Then 
  \begin{eqnarray*}
    \tilde H_d(M_f^\alpha) & \cong & 
    \bigoplus_{\beta \in A_{\leq \alpha}} \bigoplus_{i = 1}^{\mu(\beta; f)}
    \tilde H_{d - d(\beta) + 1}(C(\mathcal L(X^\alpha, V^\beta)), \mathcal L(X^\alpha, V^\beta))\\
    & \cong & 
    \tilde H_d(\mathcal L(X^\alpha, \{0\})) \oplus \\
    & & 
    \bigoplus_{\beta \in A_{\leq \alpha}\setminus\{0\}} \bigoplus_{i = 1}^{\mu'(\beta; f)}
    \tilde H_{d - d(\beta) + 1}(C(\mathcal L(X^\alpha, V^\beta)), \mathcal L(X^\alpha, V^\beta))
  \end{eqnarray*}
  According to the 
  induction hypothesis, we have $\mu(\beta; f) = \mu'(\beta; f)$ for every $\beta \in A_{\leq \alpha}$
  with $\beta \notin \{0, \alpha\}$ and, moreover, $\mu'(0; f) = 1$ 
  with a contribution of $\tilde H_d(\mathcal L(X^\alpha, \{0\}))$ to the upper direct sum.
  Counting the rank of $\tilde H_d(M^\alpha_f)$ and 
  taking into account that for $\beta = \alpha$ we get a contribution 
  $\ZZ^{\mu'(\alpha; f)}$ in the first and $\ZZ^{\mu(\alpha; f)}$ in the second line, 
  we see that indeed $\mu'(\alpha; f) = \mu(\alpha; f)$ as desired. 
\end{proof}

\subsection{Generalization to complete intersections}

\begin{definition}
  \label{def:IsolatedStratifiedCompleteInteresectionSingularity}
  Let $(X, 0) \subset (\CC^n, 0)$ be a reduced complex analytic space endowed with 
  a complex analytic Whitney stratification $X = \bigcup_{\alpha \in A} V^\alpha$.
  We say that a holomorphic function 
  \[
    f \colon (\CC^n, 0) \to (\CC^k, 0)
  \]
  has an \textit{isolated singularity in the stratified sense} on $X$, if there exists an 
  open neighborhood $0 \in U \subset \CC^n$ such that at every point 
  $0 \neq x \in U \cap X \cap f^{-1}(\{0\})$ we have 
  \[
    \D f(x) \neq 0
  \]
  on the tangent space $T_x V^\alpha$ at $x$ of the stratum $V^\alpha$ containing the point.
\end{definition}

The generalization of Tib{\u a}r's bouquet decomposition to such a function 
has already been alluded to in \cite[Section 4.3]{Tibar95}. A detailed proof for the technique 
was given in \cite{Zach18BouquetDecomp} for the special case where $(X, 0)$ 
is a product of a generic determinantal variety with some affine space. 

\begin{theorem}
  \label{thm:BouquetDecompositionsCompleteIntersections}
  In the setup of Definition \ref{def:IsolatedStratifiedCompleteInteresectionSingularity}
  the Milnor fiber admits a decomposition 
  \[
    M_{f|(X, 0)} \cong_{\mathrm{ht}} \mathcal L^{k-1}(X, \{0\}) \vee 
    \bigvee_{\alpha \in A \setminus \{0\}} \bigvee_{i = 1}^{\mu(\alpha; f)}
    S^{d(\alpha) - k + 1}(\mathcal L(X, V^\alpha))
  \]
  for some numbers $\mu(\alpha; f)$ with $\mu(\alpha; f) = 0$ for $d(\alpha) = \dim V^\alpha < k$. 
\end{theorem}

Here $\mathcal L^{k-1}(X, \{0\})$ denotes the \textit{complex link of codimension} $k-1$, 
cf. \cite{Zach21Links}. 

\begin{example}
  \label{exp:RunningExampleCompleteIntersection}
  The symmetric matrices $\CC^{2\times 2}_{\sym}$ appear as a hyperplane section of 
  the full $2\times 2$-matrices 
  \[
    \CC^{2\times 2}_\sym = \left\{
      \begin{pmatrix}
        x & w \\ z & y
      \end{pmatrix}
      \in \CC^{2\times 2} : 
      w - z = 0
    \right\}
    \subset \CC^{2\times 2}.
  \]
  Therefore, the singularity given by the function $f = y-x^k$ from 
  Examples \ref{exp:RunningExampleMorsification} and \ref{exp:RunningExampleIntroduction}
  can be given the structure of an isolated singularity for 
  \[
    g \colon (\CC^{2\times 2}, 0) \to (\CC^2, 0), \quad 
    \begin{pmatrix}
      x & w \\ z & y
    \end{pmatrix}
    \mapsto 
    \left(z - w, y - x^k\right)
  \]
  and its restriction to $X = \{ \varphi \in \CC^{2\times 2} : \det \varphi = 0\}$, 
  again considered with its rank stratification. 

  The critical locus of $g$ on $V^1 = \{\varphi \in \CC^{2\times 2} : \rank \varphi = 1\}$ 
  is given by the zero locus of the equations 
  \[
    w + z = y + k\cdot x^k = 0.
  \]
  Together with the equation $xy-zw = 0$ for the determinant, this 
  defines a curve $C \subset \CC^{2\times 2}$. 
  The geometric intersection $C \cap g^{-1}(\{0\})$ with the central 
  fiber of $g$, however, consists only of the origin, so that the requirements 
  for $g|(X, 0)$ to have an isolated singularity in the stratified sense 
  in Definition \ref{def:IsolatedStratifiedCompleteInteresectionSingularity}
  are met.

  It is evident from the simplicity of the construction of this example 
  that again $M_{g|(X, 0)} \cong_{\mathrm{ht}} \bigvee_{j=1}^k S^1$, 
  $\mathcal L^1(X, \{0\}) \cong_{\mathrm{ht}} S^1$, and therefore also
  $\mu(1;g) = k-1$.
\end{example}

Before we can prove Theorem \ref{thm:BouquetDecompositionsCompleteIntersections}, 
we need the following adaptation of a well known lemma in the context of 
ICIS.

\begin{lemma}
  \label{lem:RegularValueEtc}
  Let $\mathbb P = \mathbb P(\CC^k)$ be the space of lines in the codomain 
  of $f$ and let $L \in \mathbb P$ be a regular value of the map 
  \[
    \mathbb P f \colon X \setminus f^{-1}(\{0\}) \to \mathbb P, \quad x \mapsto [f(x)].
  \]
  Then the function 
  \[
    f' \colon (\CC^n, 0) \to (\CC^k/L, 0) \cong (\CC^{k-1}, 0), \quad x \mapsto f(x) + L
  \]
  has an isolated singularity on $(X, 0)$ in the stratified sense. 

  Suppose that $\{0\} = V^0$ is a stratum of $X$. 
  Then the stratification on the space $Z = f^{-1}(L) \subset X$ induced by 
  $\{Z \cap V^\alpha\}_{\alpha \in A}$ is Whitney regular and the restriction
  \[
    f|Z \colon (Z, 0) \to (L, 0)
  \]
  also has an isolated singularity in the stratified sense. 
\end{lemma}

\begin{proof}
  Without loss of generality, we may assume that $L$ is the $y_k$-axis in $\CC^k$ 
  so that for $h = (f_1,\dots, f_{k-1})$ we have $Z = h^{-1}(\{0\})$.
  We claim that $0 \in \CC^{k-1}$ is a regular value of $h$ on $X\setminus\{0\}$. 

  Indeed, suppose that $0 \neq x\in Z \subset X$ was an arbitrary point. For 
  $x \notin f^{-1}(\{0\})$ there is nothing to show. If $x \in f^{-1}(\{0\})$
  happens to lay on the central fiber, 
  then $\D f(x)$ has full rank on the tangent space $T_x V^\alpha$ of the stratum of $X$ 
  containing $x$. In particular this implies that also $h = (f_1, \dots, f_{k-1})$ must 
  have full rank there. 

  It follows that $(Z \setminus \{0\}) \subset (X \setminus \{0\})$ inherits the Whitney 
  stratification from $X$. This can be extended by the one-point-stratum $V^0 = \{0\}$ 
  to a full Whitney stratification of the germ $(Z, 0)$. 

  It is already clear from the construction that $f|Z$ does not have a critical point 
  on the central fiber outside the origin. 
  For points in $Z \setminus f^{-1}(\{0\})$ let $C \subset Z$ be the closure of the 
  stratified critical locus of $f|Z$.
  If $0 \in C$ was in the closure of critical points, then, according to the 
  Curve Selection Lemma, we can choose a real analytic path $\gamma \colon ([0, \varepsilon), 0) \to (C, 0)$
  with $\gamma(t) \notin f^{-1}(\{0\})$ for $t \neq 0$. 
  But then 
  \[
    f(\gamma(t)) = f(\gamma(t)) - f(0) = 
    \int_0^t \D f(\gamma(t)) \cdot \dot \gamma(t) \D t = \int_0^t 0 \D t = 0,
  \]
  a contradiction for $t \neq 0$. 
  Hence $0 \in Z$ must be an isolated critical point of $f$ on $Z$.
\end{proof}

\begin{proof}(of Theorem \ref{thm:BouquetDecompositionsCompleteIntersections})
  We prove the theorem by induction on the dimension $k$ of the codomain of $f$. 
  Choose $L$ as in Lemma \ref{lem:RegularValueEtc} so that 
  \[
    f|Z \colon (f^{-1}(L), 0) \to (L, 0)
  \]
  has an isolated singularity at $0$ on $Z = f^{-1}(L)$ in the sense of Lemma 
  \ref{lem:RegularValueEtc}.
  We may apply Theorem \ref{thm:BouquetDecompositionHypersurface} and obtain 
  a decomposition 
  \[
    M_{f|Z} \cong \mathcal L(Z, \{0\}) \vee 
    \bigvee_{\alpha \in A} \bigvee_{i = 1}^{\mu(\alpha; f|Z)} S^{d(\alpha) - k + 1}(\mathcal L(Z, Z\cap V^\alpha)).
  \]
  In particular $\mu(\alpha; f|Z) = 0$ whenever $d(\alpha) = \dim V^\alpha < k$ 
  so that the space $V^\alpha \cap Z$ is $V^0 = \{0\}$ or empty. 
  Furthermore, $\mathcal L(Z, Z\cap V^\alpha) = \mathcal L(X, V^\alpha)$ 
  since on all strata $V^\alpha$ meeting $Z$ non-trivially, the intersection 
  is transversal. 

  The complex link $\mathcal L(Z, \{0\})$ is the Milnor fiber of a generic 
  linear form $l \colon (Z, 0) \to (\CC, 0)$. If we assume that $L$ 
  was the $y_k$-axis, then this naturally agrees with the Milnor fiber
  of 
  \[
    h = (f_1,\dots, f_{k-1}, l) \colon (X, 0) \to (\CC^{k-1} \oplus \CC^1, 0)
  \]
  and, hence, also the Milnor fiber of 
  \[
    \tilde f = (f_1,\dots, f_{k-1}) \colon (X \cap l^{-1}(\{0\}), 0) \to (\CC^{k-1}, 0)
  \]
  We may therefore proceed by induction with $\tilde f$ in place of $f$. 
\end{proof}

\begin{corollary}
  \label{cor:MilnorNumbersCompleteIntersections}
  For a sufficiently generic choice of coordinates $(y_1,\dots, y_k)$ of the codomain 
  of $f$ and linear forms $l_1,\dots, l_k \colon \CC^n \to \CC$, 
  the numbers $\mu(\alpha; f)$ in Theorem \ref{thm:BouquetDecompositionsCompleteIntersections} 
  can be computed as 
  \begin{equation}
    \label{eqn:MilnorNumbersCompleteIntersections}
    \mu(\alpha; f) = \sum_{j = 1}^k \mu\left(\alpha; f_j|(X \cap \{f_{1} = \dots = f_{j-1} = 0\} \cap \{l_k = \dots = l_{j+1} = 0\})\right).
  \end{equation}
\end{corollary}

\begin{example}
  \label{exp:RunningExampleMilnorNumbers}
  In Example \ref{exp:RunningExampleCompleteIntersection} we have two summands 
  in (\ref{eqn:MilnorNumbersCompleteIntersections}) for $f_1 = l_1 = z - w$ 
  and $f_2 = y - z^k$. For the second linear form, we may again choose $l_2 = x + y$.

  For $j=1$ we find $X \cap \{f_1 = 0\}$ which recovers the symmetric $2\times 2$-matrices 
  of rank $\leq 1$. In this case it is clear that the strata of the rank stratification 
  for symmetric matrices is given by the intersection of the rank strata with the linear 
  subspace of symmetric matrices, but Lemma \ref{lem:RegularValueEtc} assures that we 
  can always achieve this situation if we need to. 

  The number $\mu(1; f_2|(X \cap \{f_1 = 0\}))$ must then be equal to $k-1$ as before 
  for the Morsification by $l_2 = x+y$. Then $\mathcal L(Z, \{0\})$ in the 
  corresponding induction step of the proof of Theorem 
  \ref{thm:BouquetDecompositionsCompleteIntersections} is the Milnor 
  fiber of $l_2$ on $(X \cap \{z-w = 0\}, 0)$. 

  The summand in (\ref{eqn:MilnorNumbersCompleteIntersections}) for $j=2$ would compute 
  the Morse critical points of a Morsification of $f_1 = z - w$ on 
  $V^1 \cap \{x + y = 0\}$. But since $f_1$ is already a linear form in general position, 
  there are no such critical points and the summand is zero. 
  Therefore $\mu(1; g) = k-1 + 0 = k-1$ as expected.
\end{example}

\section{Obstructions on fibers of maps}
\label{sec:TopologicalObstructionsOnFibersOfMaps}

Throughout this section we suppose $U \subset \CC^n$ is 
open and $X \subset U$ is a reduced complex analytic space, 
endowed with a Whitney $a$-regular stratification $X = \bigcup_{\alpha \in A} V^\alpha$. 

\subsection{Collections of $1$-forms on stratified spaces}

\begin{definition}
  \label{def:DegeneracyPointCollectionOf1Forms}
  Let $\omega_1,\dots, \omega_m \in \Gamma(U, \Omega^1_U)$ be a collection of $1$-forms. 
  We say that $x \in X$ is a \textit{degeneracy point} of $\omega = (\omega_1,\dots, \omega_m)$ 
  if the $\omega_i$ are linearly dependent in the dual $T_x^*V^\alpha$ of the tangent space 
  to the stratum $V^\alpha$ containing $x$.
\end{definition}

Note that the $a$-regularity of the stratification assures that for any given collection 
of $1$-forms $\omega$ as in Definition \ref{def:DegeneracyPointCollectionOf1Forms}
the set of degeneracy points is not only locally closed, but actually closed. 
We denote the set of degeneracy points of $\omega$ on $X$ by $D(\omega; X)$. 

\subsubsection{Obstruction theory}
We briefly recall the required notions for what follows. The reader is 
referred to \cite{Brasselet22} and the references given there for details.

Let $(K, \partial K)$ be an oriented simplicial complex of real dimension $n$ 
with an oriented subcomplex $\partial K$ of real dimension $<n$ and let 
$E \to K$ be a complex vector bundle of rank $r$. A collection of sections 
$s_1,\dots, s_m$ in $E$ is called an $m$-\textit{frame} in $E$ if the vectors 
$s_1(p), \dots, s_m(p)$ are linearly independent in the fiber $E_p$ at all points 
$p \in K$. 

For any $n$-simplex $(\Delta, \partial \Delta)$ in $K$ and an 
$m$-frame $s = (s_1,\dots, s_m)$ in $E$ over $\partial \Delta$, 
one can try to extend $s$ to an $m$-frame in $E$ over the interior of $\Delta$. 
Whenever $n - 1 < 2r - 2m + 1$ is smaller than the \textit{obstruction dimension}, 
this is always possible. In case $n -1  = 2r - 2m + 1$ there is a well-defined 
\textit{topological obstruction} to extending $s$ as an $m$-frame 
from $\partial \Delta$ to the interior of $\Delta$. This obstruction is 
an integer and it can be defined as the evaluation of the obstruction 
cocycle associated to $s$ over $\partial \Delta$ on the 
orientation class of $(\Delta, \partial \Delta)$. An alternative description 
of this number is given by the degree of the map 
\[
  s|_\Delta \colon \partial \Delta \to \Stief(m, r)
\]
to the Stiefel manifold of $m$-frames in $\CC^r$. More precisely, 
we choose a local trivialization $E|_{\Delta} \cong \CC^r \times \Delta$ of 
$E$ over $\Delta$ so that $s$ takes any point $p\in \Delta$ to an $m$-frame 
in $\CC^r$. The homology group $H_{2r - 2m + 1}(\Stief(m, r)) \cong \ZZ^1$ 
is cyclic with a distinguished generator $\gamma$ of positive orientation and the 
degree of $s|_\Delta$ is defined as the coefficient $d$ of the image of the 
orientation class 
\[
  \left(s|_\Delta\right)_*([\partial \Delta]) = d \cdot \gamma.
\]
We refer to e.g. \cite{EbelingGuseinZade05CharNumbers} for details, where 
the authors call this the \textit{index} of the $m$-frame.

\medskip
This notion of obstruction is naturally additive on the simplexes $(\Delta, \partial \Delta)$ 
of $(K, \partial K)$ in the following sense. Given an $m$-frame $s$ in $E$ over 
$\partial K$, we can extend $s$ over the subcomplex $K' \subset K$ given by all 
simplices of dimension $<n$. 
Now on every single $n$-dimensional simplex $\Delta$, we may assume to have 
a prescribed extension $s|_{\partial \Delta}$ on the boundary with a local 
obstruction on $(\Delta, \partial \Delta)$. One can then define the total obstruction 
to extending $s$ from $\partial K$ to $K$ as the sum of the obstructions on the 
top-dimensional simplices $(\Delta, \partial \Delta)$. 

\medskip
It is a well-known fact that when $n - 1 = 2r - 2m + 1$ is equal to the obstruction 
dimension, then the obstruction to extending $s$ as an $m$-frame 
from $\partial K$ to $K$ counts the number of isolated degeneracy points of 
a generic extension $\tilde s$ of $s$ as a collection of sections in $E$ over $K$. 
See e.g. \cite[Proposition 2]{EbelingGuseinZade05CollectionsOf1Forms} for a variant 
of this. For our purposes, this will also follow from the ``Morsification Lemma'', 
Lemma \ref{lem:Morsification}, below and an inspection of the local obstructions 
at such isolated degeneracy points, together with the aforementioned additivity.  

\medskip
At last, we would like to remind the reader that every compact real analytic space 
can be triangulated (see e.g. \cite{Lojasiewicz64}) and that, moreover, a
reduced complex analytic space $X$ always carries a unique orientation class. 
The latter is clear in the smooth case, but a little more care is needed for 
singular varieties. There, it basically follows from the fact that the singular 
locus has real codimension $\geq 2$ so that the orientation class 
of $X_\reg$ always extends naturally to an orientation class of $X$.


\subsubsection{Nash modifications}
For any stratum $V^\alpha$ we may consider its closure $X_\alpha = \overline{V^\alpha}$ 
and the \textit{Nash modification}
\[
  \nu_\alpha \colon \tilde X_\alpha \to X_\alpha.
\]
It is defined as the closure of the graph of the Gauss map 
\[
  \gamma \colon V^\alpha \to \Grass(d(\alpha), n), \quad p \mapsto [T_p V^\alpha \subset T_p \CC^n]
\]
with $\nu_\alpha$ induced by the natural projection 
$\Grass(d(\alpha), n) \times X_\alpha \to X_\alpha$.
On $\tilde X_\alpha$ we find the \textit{Nash bundle} $\tilde T$, 
the restriction of the tautological bundle on $\Grass(d(\alpha), n)$ to $\tilde X_\alpha$. 

We denote the dual of the Nash bundle by $\tilde \Omega^1$.
For a $1$-form $\omega \in \Gamma(U, \Omega^1_{\CC^n})$ there is a natural 
notion of pullback 
$\nu_\alpha^* \omega \in \Gamma(\nu_\alpha^{-1}(U), \tilde \Omega^1)$
where at a point $(T, p) \in \tilde X_\alpha$ the section $\nu_\alpha^* \omega$ 
is given by the map 
\[
  T \to \CC, \quad v \mapsto \langle \omega(p), v \rangle.
\]
Here we denote by $\langle \cdot, \cdot \rangle$ the natural pairing between a vector 
space and its dual. 

By construction, $\nu_\alpha$ is an isomorphism over $V^\alpha$ with 
$\nu_\alpha^* TV^\alpha \cong \tilde T$ there. 
For points $p \in X_\alpha \setminus V^\alpha$ in the boundary of $V^\alpha$ 
the fiber of $\nu_\alpha$ consists of pairs 
\[
  \{(T, p) \in \Grass(d(\alpha, n)) \times X_\alpha : 
  T \textnormal{ is a limiting tangent space to } V^\alpha \textnormal{ at } p \}.
\]
In general we have little control over these sets. Note, however, that
\[
  \dim \nu_\alpha^{-1}(X_\alpha \setminus V^\alpha) < \dim X_\alpha = \dim V^\alpha
\]
since $\tilde X_\alpha$ contains $V^\alpha$ as a dense open subset. 


\subsection{The main technical result}

As before let $U \subset \CC^n$ be open, $X \subset U$ an 
equidimensional complex analytic space 
endowed with a Whitney stratification $X = \bigcup_{\alpha \in A} V^\alpha$ with 
irreducible strata and suppose that $0 \in X \subset \CC^n$ with $V^0 = \{0\}$ 
one of the strata. 

Furthermore, let
\[
  f \colon X \to \CC^k, \quad f(0) = 0
\]
be a holomorphic function and $\omega = (\omega_1,\dots, \omega_m)$ a 
collection of $1$-forms $\omega_j \in \Gamma(U, \Omega^1_{\CC^n})$ such that 
the intersection 
\begin{equation}
  \label{eqn:IntersectionAssumption}
  f^{-1}(\{0\}) \cap D(\omega; X) = \{0\}
\end{equation}
consists of only one single point.

We may choose a ``Milnor ball'' for the germ $(X, 0)$, i.e. a ball 
$B_\varepsilon$ of radius $\varepsilon > 0$ around $0 \in X$ such that for all 
$\varepsilon \geq \varepsilon' > 0$ the intersection 
$\partial B_{\varepsilon'} \cap X$ is transversal. 
Then $X \cap B_\varepsilon$ inherits a natural real analytic stratification 
by intersecting the interior of $B_\varepsilon$ and its boundary with the strata $V^\alpha$ 
of $X$. 

Consider a regular value $v \in \CC^k$ of $f|(X \cap B_\varepsilon)$. 
Since the boundary $X \cap \partial B_\varepsilon$ and its intersection 
with the fibers of $f$ are compact, we have that 
for $1 \gg |v| \geq 0$ small enough, the boundary of the fiber 
$X \cap \partial B_\varepsilon \cap f^{-1}(\{v\})$ remains disjoint 
from the degeneracy locus $D(\omega; X)$. 

Let $\nu \colon \tilde X \to X$ be the Nash transformation and consider 
its restriction to $X \cap B_\varepsilon$. 
We let 
\[
  M_f = X \cap B_\varepsilon \cap f^{-1}(\{v\})
\]
be the fiber over the regular value $v$ and 
\[
  \tilde M_f = \nu^{-1}(M_f) \subset \tilde X
\]
its preimage under the Nash modification. By construction, 
the pullback $\nu^* \omega$ provides us with a $k$-frame 
in the dual of the Nash bundle $\tilde \Omega^1$ over the 
boundary $\partial \tilde M_f := \nu^{-1}(\partial M_f)$ 
and there is a well-defined topological obstruction 
to extending $\nu^* \omega$ as an $m$-frame to the interior of $\tilde M_f$:
\begin{equation}
  \label{eqn:TopologicaObstruction}
  \lambda = \langle \mathrm{Obs}(\nu^*\omega), [\tilde M_f]\rangle.
\end{equation}
A priori, this number seems to depend on various choices, such as the 
radius $\varepsilon$ of the ball and the regular value $v$. 
However, we will see in Theorem \ref{thm:MainWorkhorse}
that under some further mildly restrictive assumptions, this is not the 
case.

\begin{remark}
  \label{rem:RegularValueRestriction}
  In the setting introduced above, we would like the obstruction to 
  extending $\nu^* \omega$ from $\partial \tilde M_f$ to its interior 
  to be located only in the top-dimensional strata. Without loss of 
  generality, we may assume $X$ to be irreducible, i.e. the closure 
  of its unique open stratum $V^d$ of complex dimension $d$. 

  Since $v$ is a regular value of $f|(X \cap B_\varepsilon)$, it is 
  evident that $M_f \cap X \cap \partial B_\varepsilon$ is of 
  real dimension $<2d$. 
  Unfortunately, for $\tilde M_f \cap \nu^{-1}(X \cap \partial B_\varepsilon)$
  this need no longer be the case: A priori, $\tilde M_f$ could be reducible 
  with additional components over the complement of the smooth stratum $V^d$.
  Then the obstruction to extending $\nu^*\omega$ to the interior of $\tilde M_f$
  could potentially also take non-trivial values there, i.e. outside 
  the open subset $\nu^{-1}(V^d)$.

  In order to exclude such ``phantom obstructions'', we need to make the following 
  restrictions. We refine the natural decomposition 
  \[
    \tilde X = \bigcup_{\alpha \in A} \nu^{-1}(V^\alpha)
  \]
  to a Whitney regular stratification 
  of $\tilde X$ in its new ambient space. Since $\nu$ is an isomorphism on the smooth 
  locus $X_{\reg}$, we can do so without altering $\nu^{-1}(X_\reg)$. If we choose
  $v$ to be a regular value of 
  \[
    f \circ \nu \colon \tilde X \cap \nu^{-1}(B_\varepsilon) \to \CC^k
  \]
  with respect to this stratification, then it is easy to see that indeed 
  $\nu^{-1}(X_\reg)$ is again dense in $\tilde M_f$ -- just as already anticipated 
  by intuition.
\end{remark}

The following will now be our main workhorse. Since we expect this to be of some interest 
of its own in this generality, we state it as an independent theorem.

\begin{theorem}
  \label{thm:MainWorkhorse}
  For every sufficiently small $\varepsilon > 0$ such that $B_\varepsilon$ is a Milnor 
  ball for $(X, 0)$, there exists $\delta > 0$ and a dense open subset 
  $\Omega \subset D_\delta \subset \CC^k$ of regular values of $f|(X \cap B_\varepsilon)$ 
  such that 
  for every $v \in \Omega$ 
  the topological obstruction to extending $\nu^*\omega$ as an $m$-frame 
  from $\nu^{-1}(X \cap \partial B_\varepsilon \cap f^{-1}(\{v\}))$ to the interior 
  of $\nu^{-1}(X \cap B_\varepsilon \cap f^{-1}(\{v\}))$ is equal to 
  \[
    \lambda = 
    \chi\left(R \nu_* \left(\mathcal Kosz(\nu^*f) \otimes \mathcal ENC(\nu^*\omega)\right)\right).
  \]
  In particular, this number is non-negative.
\end{theorem}

\medskip
Here 
$\mathcal Kosz(\nu^* f)$ and 
$\mathcal ENC(\nu^* \omega)$ are certain complexes of 
sheaves associated to a collection of functions and 
$1$-forms which will be described below. The additional restriction made for the 
regular values $v$ to be elements of $\Omega$ is due to the reasons explained in 
Remark \ref{rem:RegularValueRestriction}.

\medskip
The formula for $\lambda$ in Theorem \ref{thm:MainWorkhorse} depends on three ingredients:
The space $(X, 0)$, the function $f$, and the collection of $1$-forms $\omega$. 
Note that $(X, 0)$ alone determines the Nash transformation and all the vector bundles 
involved in the complex of sheaves $\mathcal Kosz(\nu^*f) \otimes \mathcal ENC(\nu^*\omega)$.
The pair $(f, \omega)$ only introduces maps between them. 
Since the derived pushforward is functorial, this allows us to pre-compute the direct 
image complex as a function of $f$ and $\omega$:
\begin{corollary}
  \label{cor:bakingComplexes}
  In the setup of Theorem \ref{thm:MainWorkhorse} there exists a finite complex of free 
  modules 
  $\mathrm{Prep}(X; k, m)$ 
  over the ring
  \[
    \OO_{\CC^n, 0}[c_l, a_{i, j} : l = 1,\dots, k, \, i = 1, \dots, m, \, j = 1, \dots, n]
  \]
  such that for every function $f \colon (\CC^n, 0) \to (\CC^k, 0)$ and every collection 
  of $1$-forms $\omega = \omega_1, \dots, \omega_m$, $\omega_i \in \Omega^1_{\CC^n, 0}$, 
  we have that 
  \[
    R \nu_* \left(\mathcal Kosz(\nu^*f) \otimes \mathcal ENC(\nu^*\omega)\right)
  \]
  is quasi-isomorphic to the complex obtained from $\mathrm{Prep}(X; k, m)$ 
  by substituting $f$ for $c$ and the components $\omega_{i, j}$ for the $a_{i, j}$.
\end{corollary}

\begin{remark}
  Note that we have not made any non-degeneracy conditions on the pair $(f, \omega)$ 
  in the formulation of Corollary \ref{cor:bakingComplexes}. The complex 
  $\mathrm{Prep}(X; k, m)$ always exists and can be considered as an analogue 
  of the complexes used to define the classical ``homological index'', e. g. 
  in \cite{GomezMont98}. Whether or not these complexes have finitely supported 
  cohomology after substitution of $(c, a)$ with $(f, \omega)$ will depend 
  on the (non-)degeneracy of the latter on $(X, 0)$.
\end{remark}

\subsection{Geometric interpretations of some commutative algebra}
\subsubsection{The Koszul complex}

Let $R$ be a commutative ring and $G \cong R^n$ a free $R$-module with basis $\{e_i\}_{i = 1}^n$.
Denote the $R$-dual of $G$ by $G^\vee = \Hom_R(G, R) \cong R^n$ and let $\{\varepsilon_i\}_{i=1}^n$ be the dual basis. 
To an element $\varphi = \sum_{i=1}^n \varphi_i \cdot \varepsilon_i \in G^\vee$ we associate the 
\textit{Koszul complex}:
\begin{equation}
  \label{eqn:KoszulComplexDefinition}
  \Kosz(\varphi) \colon 
  \begin{xy}
    \xymatrix{
      0 \ar[r]&
      \bigwedge^n G \ar[r]^{\varphi \invneg } &
      \bigwedge^{n-1} G \ar[r]^{\varphi \invneg } &
      \dots \ar[r]^{\varphi \invneg } &
      \bigwedge^1 G \ar[r]^{\varphi \invneg } &
      \bigwedge^0 G \ar[r] &
      0
    }
  \end{xy}
\end{equation}
with maps given by 
\[
  e_I = e_{i_1} \wedge \dots \wedge e_{i_p} \mapsto \varphi \invneg e_I = \sum_{k = 1}^p (-1)^{k+1} \varphi_{i_k} \cdot e_{i_1} \wedge \dots \wedge \widehat{e_{i_k}} \wedge \dots \wedge e_{i_p}
\]
where the $\widehat{\cdot}$ indicates that the term should be omitted. We will also write $e_{I \setminus \{i_k\}}$ for the $k$-th such generator.

\begin{remark}
  In the literature it is customary to introduce the Koszul complex for a sequence 
  of elements $a_1,\dots,a_n \in R$. We will write 
  \[
    \Kosz(a_1,\dots, a_n) := \Kosz(\varphi)
  \]
  for the the linear form $\varphi = \sum_{i=1}^n a_i \cdot \varepsilon_i$ on a free
  module $G = R^n$.
\end{remark}

\begin{lemma}
  \label{lem:FunctorialityKoszulComplex}
  Let $F \cong R^m$ and $G \cong R^n$ be free modules, $\varphi \in G^\vee$, 
  and $f \colon F \to G$
  a morphism. Then we obtain a natural morphism of complexes 
  $f_* \colon \Kosz(f^*\varphi) \to \Kosz(\varphi)$ from the induced maps 
  $f^{\wedge p} \colon \bigwedge^p F \to \bigwedge^p G$ on the exterior powers 
  and $f^* \varphi \in F^\vee$.
\end{lemma}

\noindent
We leave this straightforward computation to the reader. 

\begin{theorem}
  \label{thm:KoszulComplexFromSection}
  Let $\mathcal E$ be a locally free sheaf of $\OO_X$-modules of rank $n$ on a ringed space $(X, \OO_X)$ 
  and $\varphi \in \Gamma(X, \mathcal E^\vee)$ a section in its dual. Then we obtain a 
  natural complex of sheaves 
  \[
    \mathcal K\mathrm{osz}(\varphi) \colon 
    \begin{xy}
      \xymatrix{
        0 \ar[r]&
        \bigwedge^n \mathcal E \ar[r]^{\varphi\invneg} &
        \bigwedge^{n-1} \mathcal E \ar[r]^{\varphi\invneg} &
        \dots \ar[r]^{\varphi\invneg} &
        \bigwedge^1 \mathcal E \ar[r]^{\varphi\invneg} &
        \bigwedge^0 \mathcal E \ar[r] &
        0.
      }
    \end{xy}
  \]
\end{theorem}

\begin{proof}
  It is clear that a local trivialization of $\mathcal E$ and its dual gives rise to a Koszul complex. That such 
  local models for $\mathcal K\mathrm{osz}(\varphi)$ glue naturally on the overlaps
  of any two local trivializations follows from Lemma \ref{lem:FunctorialityKoszulComplex}.
\end{proof}

\begin{remark}
  In \cite{Zach22} we have worked with the Koszul complex which arises from taking successive wedge products 
  with an element $v \in R^n$ in a free module. This is also functorial in the above sense, 
  but it does not
  provide a natural resolution of the quotient $R/\langle v_1,\dots, v_n\rangle$ since one has to choose
  an isomorphism $\bigwedge^n R^n \cong R$ for the augmentation map. These choices do not glue naturally 
  and therefore we prefer to work with the Koszul complex as introduced above in this note.
\end{remark}

\subsubsection{The Eagon-Northcott complex}

We briefly recall the construction of the Eagon-Northcott complex \cite{EagonNorthcott62}, 
thereby fixing notation for what follows. 

Let $A \in R^{k \times n}$ be a matrix with entries in a commutative ring $R$. We consider the rows $A_i$ 
of $A$ to be elements in the dual of a free $R$-module $G \cong R^n$, i.e. if we 
denote by 
$\{e_j\}_{j = 1}^n$ the canonical basis of $R^n$ and by $\varepsilon_j$ their respective dual 
elements, then
\[
  A_i = \sum_{j=1}^n a_{i, j} \cdot \varepsilon_j \in G^\vee.
\]
Let $S = R[x_1,\dots,x_k]$ be the polynomial ring in $k$ variables over $R$ and denote the homogeneous 
part of degree $d$ by $S_d$. The terms of the Eagon-Northcott complex are
\begin{equation}
  \label{eqn:TermsENC}
  \mathrm{ENC}(A)_p = 
  \begin{cases}
    R & \textrm{ for } p = 0 \\
    S_{p-1} \otimes_R \bigwedge^{k+p-1} G & \textrm{ for } p > 0.
  \end{cases}
\end{equation}
The morphisms between these are given by 
\[
  S_0 \otimes_R \bigwedge^k G \to R, \quad 1 \otimes e_I \mapsto \det A_{\{1,\dots, k\}, I}
\]
for $p = 1$ and 
\begin{eqnarray*}
  S_d \otimes_R \bigwedge^{k+d} G &\to & S_{d-1} \otimes_R \bigwedge^{k+d-1} G, \\
  x_1^{\alpha_1} \cdots x_k^{\alpha_k} \otimes e_I &\mapsto &
  \sum_{i=1}^k 
  \begin{cases}
    \frac{x^\alpha}{x_i} \otimes A_i \invneg e_I & \textnormal{ if } \alpha_i > 0 \\
    0 & \textnormal{otherwise}
  \end{cases}
\end{eqnarray*}
for $p > 1$. Here we denote by $A_{I, J}$ the submatrix of $A$ specified by the indices 
in $I$ and $J$, respectively.

\medskip
We would like to highlight some properties of the Eagon-Northcott complex 
which have probably been known to the the authors of \cite{EagonNorthcott62} 
but have not been spelled out to the full extent. 

\begin{lemma}
  \label{lem:FunctorialityENC}
  Let $f \colon F \to G$ be a morphism of free modules
  as in Lemma \ref{lem:FunctorialityKoszulComplex} 
  and $A \in R^{k \times n}$ a matrix whose rows are elements of $G^\vee$.
  Then we obtain a natural morphism 
  of complexes $f_* \colon \mathrm{ENC}(f^*A) \to \mathrm{ENC}(A)$ for the induced 
  morphisms 
  \[
    \mathrm{id} \otimes f^{\wedge (k+d)} \colon S_d \otimes \bigwedge^{k+d} F \to 
    S_d \otimes \bigwedge^{k+d} G
  \]
  on the terms of the Eagon-Northcott complexes associated to the matrices $A$ and 
  $f^*A$, respectively. 
\end{lemma}

\noindent
Again, the verification of Lemma \ref{lem:FunctorialityENC}
is a straightforward computation which we leave to the reader. 

\begin{theorem}
  \label{thm:SheafVersionENC}
  Let $\mathcal E$ be a locally free sheaf of $\OO_X$-modules of rank $n$ on a ringed space $(X, \OO_X)$ 
  and $\Phi = (\varphi_1,\dots,\varphi_k)$ a collection of 
  sections $\varphi_i \in \Gamma(X, \mathcal E^\vee)$. 
  Then we obtain a natural complex of sheaves 
  $\mathcal ENC(\Phi)$. 
%
\end{theorem}

\begin{proof}
  Similar to Theorem \ref{thm:KoszulComplexFromSection}
  this follows directly from Lemma \ref{lem:FunctorialityENC} applied 
  to the change of basis for $\mathcal E$ on the overlap of two different local 
  trivializations.
\end{proof}

Let $A \in R^{(k+1)\times (n+1)}$ be a matrix, $k \leq n$, and suppose 
that one entry of $A$ is a unit. Without loss of generality, we assume 
this entry to be $a_{k+1, n+1}$. With an appropriate change of basis 
and an application of Lemma \ref{lem:FunctorialityENC}, we may furthermore 
assume $A$ to be of the form 
\[
  A = 
  \left(
  \begin{array}{ccc|c}
    a_{1,1} & \cdots & a_{1, n} & a_{1, n+1} \\
    \vdots & & \vdots & \vdots \\
    a_{k,1} & \cdots & a_{k, n} & a_{k, n+1} \\
    \hline
    0 & \cdots & 0 & 1
  \end{array}
  \right).
\]
Let $A' \in R^{k \times n}$ be the submatrix for the upper left block of $A$. 
\begin{lemma}
  The Eagon-Northcott complex for $A$ is homotopy equivalent to the 
  complex for $A'$. 
  \label{lem:HomotopyENC}
\end{lemma}

\begin{proof}
  The free module $F = R^{n+1}$ decomposes as $F \cong F' \oplus R \cdot e_{n+1}$ 
  which induces a splitting on its exterior powers 
  \[
    \bigwedge^p F \cong \left(\bigwedge^p F'\right)
    \oplus \left(R^1 \otimes \bigwedge^{p-1} F'\right)
  \]
  If we let $S' = R[x_1,\dots,x_k] \subset S = R[x_1,\dots,x_{n+1}]$, then 
  \[
    S_d \cong \bigoplus_{j=0}^d S'_{d-j} \cdot x_{n+1}^j.
  \]
  The terms in the Eagon-Northcott complex split accordingly as 
  \begin{equation}
    \label{eqn:DirectSumDecompositionENC}
    S_d \otimes \bigwedge^{m+1-d} F \cong 
    \left(\bigoplus_{j=0}^d S'_{d-j} \cdot x_{n+1}^j \otimes \bigwedge^{m+1-d} F'\right)
    \oplus
    \left(\bigoplus_{j=0}^d S'_{d-j} \cdot x_{n+1}^j \otimes R^1 \otimes \bigwedge^{m-d} F'\right).
  \end{equation}
  Note that in homological degree $p>0$ the summands 
  \[
    S'_{p-1} \otimes R^1 \otimes \bigwedge^{m-p+1} F' \cong S'_{p-1} \otimes \bigwedge^{m-p+1} F'
  \]
  provide an inclusion of $\mathrm{ENC}(A')$ into $\mathrm{ENC}(A)$. 
  Denote this inclusion by $\iota$ and the 
  corresponding projection by $\pi$. We have to specify a homotopy of 
  $\mathrm{ENC}(A)$ by means of maps 
  \[
    h \colon \mathrm{ENC}(A)_p \to \mathrm{ENC}(A)_{p+1}.
  \]
  For $p=0$ we choose $h$ to be the zero map. In case $p>0$ we define $h$ to be zero on 
  the second set of summands in (\ref{eqn:DirectSumDecompositionENC}). On the remainder 
  we let 
  \begin{eqnarray*}
    h \colon S'_{p-j-1} \cdot x_{n+1}^j \otimes \bigwedge^{m+p-1} F' &\to &
    S'_{p-j} \cdot x_{n+1}^{j+1} \otimes R^1 \otimes \bigwedge^{m+p-1} F',\\
    x^\alpha \cdot x_{n+1}^j \otimes v &\mapsto & x^\alpha \cdot x_{n+1}^{j+1} \otimes e_{n+1} \otimes v.
  \end{eqnarray*}
  We leave it to the reader to verify that indeed $h \circ \partial + \partial \circ h = \id - \iota \circ \pi$.
\end{proof}

In the context of vector bundles this allows for the following interpretation of 
the Eagon-Northcott complex. 

\begin{theorem}
  \label{thm:EagonNorthcottSubbundles}
  Let $X$ be a Haussdorff topological space and $E \cong \CC^r \times X \to X$ a 
  trivial complex vector 
  bundle of rank $r$. Suppose $\varphi = (\varphi_1,\dots, \varphi_t)$ 
  is a $t$-frame in $E^\vee$. Then for any additional set of sections 
  $\psi = (\psi_1,\dots,\psi_k)$ in $E^\vee$, the Eagon-Northcott complex 
  for the composed collection $(\psi | \varphi)$
  is homotopy equivalent to the Eagon-Northcott complex induced 
  by the restriction of $\psi$ to the kernel 
  sub-bundle $F \hookrightarrow E$ of $\varphi$. 
\end{theorem}

\begin{proof}
  Let $p \in X$ be an arbitrary point. By assumption $\varphi$ is a $t$-frame, so there 
  is at least one $t\times t$-minor of the $t \times r$-matrix associated to $\varphi$
  which does not vanish on some open neighborhood $U$ of $p$. 
  A repeated application of Lemma \ref{lem:HomotopyENC} 
  shows that over $U$ the Eagon-Northcott complex for the extended matrix 
  \[
    \begin{pmatrix}
      \psi_{1,1} & \dots & \psi_{1, r} \\
      \vdots & & \vdots \\
      \psi_{k,1} & \dots & \psi_{k, r} \\
      \hline
      \varphi_{1,1} & \dots & \varphi_{1, r} \\
      \vdots & & \vdots \\
      \varphi_{t,1} & \dots & \varphi_{t, r} \\
    \end{pmatrix}
  \]
  is homotopy equivalent to the Eagon-Northcott complex associated to the 
  $k$ sections induced by $\psi$ in the dual of the kernel bundle of $\varphi$.
  It is evident from the geometric nature of this construction that it glues 
  over different open subsets corresponding to different sets of non-vanishing 
  minors.
\end{proof}

Note that in the above theorem we have not specified whether we have a holomorphic, 
a smooth, or only a continuous $t$-frame. As we only use the usual arithmetic in the 
respective sheaf of rings, the theorem holds true in either category. 

\begin{corollary}
  \label{cor:ENCExactForFrame}
  The Eagon-Northcott complex associated to a $t$-frame is exact.
\end{corollary}

\begin{proof} This is the content of \cite[Proposition 3.1, i)]{EagonNorthcott62}. 
  However, with the above preparations it now also follows directly 
  from an iterated application of Lemma \ref{lem:HomotopyENC} as in the proof of 
  Theorem \ref{thm:EagonNorthcottSubbundles}.
\end{proof}

\subsection{Proof of Theorem \ref{thm:MainWorkhorse}}

We first note that the cohomology of the complex of sheaves 
$\mathcal Kosz(\nu^*f) \otimes \mathcal ENC(\nu^* \omega)$
is supported in the central fiber $\nu^{-1}(\{0\})$. 
To see this, let $0 \neq q \in X \cap B_\varepsilon$ be an arbitrary point 
and $p \in \nu^{-1}(\{q\})$ a point in its preimage. 

Suppose $\nu^* f(p) \neq 0$. Then $\mathcal Kosz(\nu^*f)_p$ is exact 
and so then is its tensor product with $\mathcal ENC(\nu^*\omega)$. 
Now if $p$ comes to lay on the preimage of the central fiber of $f$ under 
$\nu$, the point $q = \nu(p)$ can not be a degeneracy point of $\omega$
due to the assumptions made in (\ref{eqn:IntersectionAssumption}). 
Consequently, the sections $\nu^*\omega$ are linearly independent 
at $p$ and $\mathcal ENC(\nu^* \omega)_p$ is exact according to 
Corollary \ref{cor:ENCExactForFrame}. 
Shrinking $X$ if necessary, it follows that 
the complex 
\[
  R\nu_* \left(\mathcal Kosz(\nu^*f) \otimes \mathcal ENC(\nu^* \omega)\right)
\]
has homology supported at most at the origin $0 \in X$.

\medskip

Now let $\underline c = (c_1,\dots, c_k)$ and 
$\underline a = (a_{i, j})$, $i=1,\dots, m$, $j = 1,\dots, n$ be 
indeterminates. We regard them as parameters in an unfolding of 
both $f$ and $\omega$ on the trivial family 
\[
  X \times \CC^k \times \CC^{m \times n} \to \CC^k \times \CC^{m \times n}.
\]
In other words, over a point $(c, a)$ in the parameter space we consider the 
function 
\[
  f^c \colon X \to \CC^k, \quad x \mapsto f(x) + c
\]
and the collection of $1$-forms $\omega^a$ given by
\[
  \omega^a_i + \sum_{j=1}^n a_{i, j} \cdot \D x_j.
\]
The respective pullbacks induce a family of complexes of coherent sheaves
\begin{equation}
  \label{eqn:FamilyOfComplexesOnNashTransform}
  \mathcal Kosz(\nu^* f^c) \otimes \mathcal ENC(\nu^*\omega^a)
\end{equation}
on the Nash transform $\tilde X \times (\CC^k \times \CC^{m \times n})$. 
Note that the sheaves of modules in this family are all locally 
free and independent of the deformation parameters $(c, a)$. 
It is only the morphisms between them which vary. 

Taking derived pushforward along $\nu$, we obtain a family of 
complexes of coherent sheaves 
\begin{equation}
  \label{eqn:PushforwardFamily}
  R \nu_* \left(\mathcal Kosz(\nu^*f^c) \otimes \mathcal ENC(\nu^*\omega^a)\right)
\end{equation}
on $X \times (\CC^k \times \CC^{m \times n})$. For the stalk at $0 \in X$ 
we can again choose a representative of this complex so that the modules 
involved are constant with respect to the deformation parameters and only 
the maps vary; for instance, one can take a strand of a suitably truncated 
\v{C}ech-complex as in \cite{Zach22}.

\medskip
As already done in \cite{Zach22}, we may at this point invoke the main 
theorem of \cite{GiraldoGomezMont02}:
The law of conservation of number for local Euler characteristics. 
It asserts that there exists a product $U' \times T$ of neighborhoods 
of $0$ in $X$ and $(0, 0)$ in $\CC^k \times \CC^{m \times n}$, respectively,
such that for every $(c, a) \in T$ we have 
\begin{equation}
  \label{eqn:ConservationPrinicpleApplied}
  \chi \left(R \nu_* \left(
  \begin{array}{c}
    \mathcal Kosz(\nu^*f) \\
    \otimes \\
    \mathcal ENC(\nu^*\omega)
  \end{array}
  \right)\right)_0
  = \sum_{p \in U'}
  \chi \left(R \nu_* \left(
  \begin{array}{c}
    \mathcal Kosz(\nu^*f^c) \\
    \otimes \\
    \mathcal ENC(\nu^*\omega^a)
  \end{array}
  \right)\right)_p
\end{equation}
This is colloquially phrased as ``The holomorphic Euler characteristic of a 
(perfect) complex of coherent sheaves satisfies the law of conservation of number''.

Choose a Milnor ball $B_\varepsilon$ for $(X, 0)$ which is small enough 
so that $B_\varepsilon \cap X$ is contained in $U'$ and suppose that 
$v \in \CC^k$ is a regular value of $f$ restricted to $B_\varepsilon \cap X$, 
with $\{v\} \times \CC^{m \times n} \cap T \ni (v, 0)$ nonempty.
In order to link this number with topological obstructions, we need 
the following technical 'Morsification lemma'.

\begin{lemma}
  \label{lem:Morsification}
  Let $(X, \partial X)$ be a compact, reduced, real analytic space with boundary,
  containing a smooth, 
  dense open subset $U \subset X$ of even real dimension $2d$ and let 
  $E \to X$ be a complex sub-bundle of rank $r$ of a trivial vector 
  bundle $\CC^n \times X \to X$. Suppose 
  \[
    A \colon X \to \CC^{m \times n}
  \]
  is an arbitrary real analytic function for some $m$ with $d \leq r - m + 1$. 

  Then there exist arbitrarily small constant matrices $B \in \CC^{m \times n}$ 
  such that the degeneracy locus 
  $D((A + B)|_E; X)$ is either empty, or a zero dimensional submanifold of $U$. 
  The latter can only happen in case of equality $d = r - m + 1$.
\end{lemma}

\begin{proof}
  We proceed by induction on $m$. For $m=1$ we let 
  \[
    N = \left\{ (p, \varphi) \in X \times \CC^{1\times n} : (A(p) + \varphi)|E_p = 0\right\}.
  \]
  Let $Z = X \setminus U$ be the complement of $U$ and denote by $N_Z$ and $N_U$ 
  the intersections of $N$ with $Z \times \CC^{1\times n}$ and $U \times \CC^{1\times n}$, 
  respectively. The fiber of $N$ over any point $p \in X$ is an affine linear subspace of 
  complex dimension $n - r$ and therefore both $N$ and $N_U$ have 
  real dimension $2d + 2n - 2r \leq 2n$. 
  On the contrary, the space $N_Z$ has real dimension $\dim_\RR Z + 2n - 2r < 2n$. 
  
  Let $\pi \colon N \to \CC^{1\times n}$ be the projection to the second factor. 
  Being of lesser dimension, the image of $N_Z$ is nowhere dense. We can choose 
  $B \in \CC^{1 \times n}$ to be any regular value of $\pi$ on the manifold $N_U$ 
  which is not contained in the image of $N_Z$. Clearly, $\pi^{-1}(\{B\})$ must 
  be empty if $d < r$ and if it is not, then it is a smooth submanifold of $U$.

  Now suppose $m > 1$ and let $a_m$ be the last row of $A$. We may apply the 
  above procedure to obtain $b_m$ arbitrarily small for $a_m$ with empty degeneracy 
  locus $D((a_m + b_m)|_E; X)$ on $X$. Consequently, the kernel bundle 
  $E' = \ker b_m \cap E$ is of rank $r-1$. We can then repeat the argument 
  for the matrix $A' \in \CC^{(m-1) \times n}$ which is obtained from $A$ 
  by deleting the last row and the vector bundle $E'$. 
\end{proof}

We can now finish the proof of Theorem \ref{thm:MainWorkhorse}. 
Since $X \cap B_\varepsilon$ is compact, the dense set of regular values 
of $f$ is also open in the Euclidean topology. 
The same holds for $f \circ \nu$ restricted to 
$\nu^{-1}(X \cap B_\varepsilon)$ as discussed in Remark \ref{rem:RegularValueRestriction}. 
We set $\Omega \subset T$ to be the intersection 
of these two sets in $T$. 

Choose $\delta > 0$ so that $D_\delta \subset T$ and $\nu^*\omega$ 
is an $m$-frame in $\tilde \Omega^1$ along 
$\nu^{-1}(X \cap \partial B_\varepsilon \cap f^{-1}(\{v\})$
for every $v \in D_\delta$. We choose $v$ to be in $D_\delta \cap \Omega$ 
and let again $M_f = f^{-1}(\{v\}) \cap X \cap B_\varepsilon$ be 
the associated fiber and $\tilde M_f = \nu^{-1}(M_f)$ its
preimage under the Nash blowup. Then $V = M_f \cap X_{\mathrm{reg}}$ 
is a dense open subset of both $M_f$ and $\tilde M_f$
and we can apply Lemma \ref{lem:Morsification} to the Nash bundle 
and the collection of $1$-forms $\nu^* \omega$ on $\tilde M_f$. 

If $d - k = \dim M_f < n - m + 1$, then there is no topological 
obstruction to extending $\nu^*\omega$ as an $m$-frame from 
$\partial \tilde M_f$ to the interior of $M_f$. Accordingly, 
Lemma \ref{lem:Morsification} tells us that the degeneracy 
locus of arbitrarily small perturbations $\omega'$ of $\nu^*\omega$ 
is empty and hence $\mathcal ENC(\omega')$ is exact along 
$\tilde M_f$. Consequently, both 
$\mathcal Kosz(\nu^*f^v) \otimes \mathcal ENC(\nu^* \omega')$ and 
its derived pushforward along $\nu$ must be trivial. 
Since $(\nu^* f^v, \omega')$ can be chosen arbitrarily close to 
$(\nu^*f, \nu^*\omega)$, it 
follows from the principle of conservation of number that 
also the Euler characteristic on the left hand side of 
(\ref{eqn:ConservationPrinicpleApplied}) must vanish. 

Suppose therefore that $d - k = n - m + 1$. The topological 
obstruction for $\nu^*\omega$ is independent of sufficiently small 
perturbations of $\nu^* \omega$ and we can choose $\omega'$ as 
in Lemma \ref{lem:Morsification} accordingly. If the degeneracy locus 
of $\omega'$ on $M_f$ is empty, then the existence of $\omega'$ shows 
that $\nu^*\omega$ can be extended as an $m$-frame to the interior 
of $M_f$. If it is non-empty, but a smooth submanifold $Z$ of the regular 
stratum of $M_f$, then it is easy to see from the additivity 
of the obstruction that every point $p \in Z$ contributes an increment 
of one to the overall obstruction. 

On the analytic side 
the sum on the right hand side of (\ref{eqn:ConservationPrinicpleApplied})
is supported on $Z$. Since $\nu$ is locally an isomorphism there, the 
derived pushforward can be dropped at such points. Following the 
argument in the proof of Lemma \ref{lem:Morsification} and 
applying Theorem \ref{thm:EagonNorthcottSubbundles}, we also see 
that at any point $p \in Z$ the complex $\mathcal ENC(\omega')_p$ 
is homotopy equivalent to a Koszul complex $\mathcal K'$ on $d - k$ local coordinates 
for $M_f$ at $p$. Altogether, $(\mathcal Kosz(\nu^* f^v) \otimes \mathcal K')_p$ then 
resolves $\mathcal O_{Z, p}$ as an $\OO_{X, p}$-module. The Euler characteristic 
of this complex is equal to one, just like the local contribution to the 
topological obstruction. This finishes the proof.

\section{Some L\^e-Greuel type formulas}

We can now provide some ways to compute 
the numbers $\mu(\alpha; f)$ 
from Theorem 
\ref{thm:BouquetDecompositionsCompleteIntersections} 
as holomorphic Euler characteristics. 

\begin{theorem}
  \label{thm:LeGreuelFormulaI}
  Let $f \colon (\CC^n, 0) \to (\CC^k, 0)$ be a holomorphic function 
  with an isolated singularity on a Whitney stratified germ 
  $(X, 0) \subset (\CC^n, 0)$, $X = \bigcup_{\alpha \in A} V^\alpha$. 
  Then there exists a change of coordinates in the codomain $(\CC^k, 0)$
  and linear forms $l_1,\dots, l_k$ on $(\CC^n, 0)$ such that
  \begin{eqnarray*}
    \mu(\alpha; f)
    & = & 
    \sum_{j=1}^k 
    \chi\left(
      R(\nu_\alpha)_*\left(
      \begin{array}{c}
        {\mathcal K}osz(\nu_\alpha^*(f_1,\dots,f_{j-1}, l_{j+1},\dots, l_k))\\
        \otimes \\
        {\mathcal E}NC(\nu_\alpha^*(\D f_1,\dots, \D f_j, \D l_{j+1},\dots, \D l_k))
      \end{array}
      \right)
    \right)\\
    & & - 
    \sum_{j=1}^k 
    \chi\left(
      R(\nu_\alpha)_*\left(
      \begin{array}{c}
        {\mathcal K}osz(\nu_\alpha^*(f_1,\dots,f_{j-1}, l_{j+1},\dots, l_k))\\
        \otimes \\
        {\mathcal E}NC(\nu_\alpha^*(\D f_1,\dots, \D f_{j-1}, \D l_{j},\dots, \D l_k))
      \end{array}
      \right)
    \right)\\
    & = & 
    \sum_{j=1}^k 
    \chi\left(
      R(\nu_\alpha)_*\left(
      \begin{array}{c}
        {\mathcal K}osz(\nu_\alpha^*(f_1,\dots,f_{j}, l_{j+1},\dots, l_k))\\
        \otimes \\
        {\mathcal E}NC(\nu_\alpha^*(\D f_1,\dots, \D f_{j}, \D l_{j},\dots, \D l_k))
      \end{array}
      \right)
    \right)\\
    & & - 
    \sum_{j=1}^k 
    \chi\left(
      R(\nu_\alpha)_*\left(
      \begin{array}{c}
        {\mathcal K}osz(\nu_\alpha^*(f_1,\dots,f_{j-1}, l_{j},\dots, l_k))\\
        \otimes \\
        {\mathcal E}NC(\nu_\alpha^*(\D f_1,\dots, \D f_{j}, \D l_{j},\dots, \D l_k))
      \end{array}
      \right)
    \right).
  \end{eqnarray*}
  If, moreover, the closures of the strata $X^\alpha = \overline{V^\alpha}$ 
  have rectified homological depth $\mathrm{rHd}(X^\alpha) = \dim(X^\alpha)$ 
  for rational coefficients and all $\alpha \in A$, then 
  \begin{eqnarray*}
    & & \sum_{\alpha \in A} \mu(\alpha; f) \cdot 
    \chi(X, V^\alpha)\\
    & = & 
    \sum_{j=1}^{k} (-1)^{j-1}\sum_{\alpha \in A}\chi\left(
      R(\nu_\alpha)_*\left(
      \begin{array}{c}
        {\mathcal K}osz(\nu_\alpha^*(f_1,\dots,f_{j-1}))\\
        \otimes \\
        {\mathcal E}NC(\nu_\alpha^*(\D f_1,\dots, \D f_{j}))
      \end{array}
    \right)
    \right)\cdot \chi(X, V^\alpha).
  \end{eqnarray*}
  where $\chi(X, V^\alpha) = \chi_{\mathrm{top}}(C(\mathcal L(X, V^\alpha)), \mathcal L(X, V^\alpha))$
  denotes the topological Euler characteristic of the pair of spaces.

  \medskip
  \noindent
  All holomorphic Euler characteristics appearing in these formulae are non-negative.
\end{theorem}

\begin{remark}
  \label{rem:ClassicalLeGreuel}
  If $(X, 0) = (\CC^n, 0)$, then the last formula transforms naturally into 
  the classical L\^e-Greuel formula from \cite{Le74} and \cite{Greuel75}. 
  There is only one stratum and the Nash modification is an isomorphism. 
  Due to dimensional reasons, both $\mathrm{Kosz}(f_1,\dots, f_{j-1})$ and 
  $\mathrm{ENC}(\D f_1,\dots, \D f_j)$ are exact but in degree $0$. 
  This reduces the computation of the Euler characteristic to the colength 
  of the ideal 
  \[
    \langle f_1,\dots, f_{j-1}\rangle + J
  \]
  where $J$ is the ideal of maximal minors of the jacobian matrix of $(f_1,\dots, f_j)$. 
\end{remark}

\begin{proof}
  Recall formula (\ref{eqn:MilnorNumbersCompleteIntersections}) 
  from Corollary \ref{cor:MilnorNumbersCompleteIntersections}:
  \[
    \mu(\alpha; f) = \sum_{j = 1}^k \mu\left(\alpha; f_j|(X \cap \{f_{1} = \dots = f_{j-1} = 0\} \cap \{l_k = \dots = l_{j+1} = 0\})\right).
  \]
  Every one of the summands on the right hand side is the number 
  of Morse points of a function $f_j$ on the \textit{central} 
  fiber of a map 
  \[
    F = (f_1,\dots,f_{j-1}, l_{j+1}, \dots, l_k) \colon (X, 0) \to (\CC^{k-1}, 0)
  \]
  which has an isolated singularity on $(X, 0)$ in the stratified sense 
  by the assumptions made using Lemma \ref{lem:RegularValueEtc}.
  Without loss of generality we may replace $(X, 0)$ by the closure of an 
  irreducible stratum $V^\alpha$ and henceforth drop $\alpha$ from the 
  notation in what follows, carrying out the proof only for the single, irreducible, 
  top-dimensional stratum.

  We may apply Theorem \ref{thm:MainWorkhorse} to $F$ and the 
  collection of $1$-forms given by 
  \[
    \omega = (\D f_1,\dots, \D f_{j-1}, \D f_{j}, \D l_{j+1}, \dots, \D l_k).
  \]
  The resulting Euler characteristic
  \[
    \chi\left(
      R(\nu_\alpha)_*\left(
      \begin{array}{c}
        {\mathcal K}osz(\nu_\alpha^*(f_1,\dots,f_{j-1}, l_{j+1},\dots, l_k))\\
        \otimes \\
        {\mathcal E}NC(\nu_\alpha^*(\D f_1,\dots, \D f_j, \D l_{j+1},\dots, \D l_k))
      \end{array}
      \right)
    \right)
  \]
  counts the number of degeneracy points of $\omega$ on a generic fiber of 
  $F$. 

  Consider the families $F^v$ and $\omega^t$ in the parameters 
  $(t, v) \in \CC \times \CC^{k-1}$ induced by 
  the unfolding 
  \[
    (f_1 - v_1, \dots, f_{j-1} - v_{j-1}, l_{j+1} - v_j, \dots, l_k - v_{k-1})
  \]
  and 
  \[
    \D f_j - t \cdot \D l_j.
  \]
  We may assume $l_j$ to be sufficiently generic as to provide a Morsification 
  of the function 
  \[
    f_j|(X \cap \{f_1 = \dots = f_{j-1} = l_{j+1} = \dots = l_k = 0\})
  \]
  which has an isolated singularity at the origin in the stratified sense according to 
  Lemma \ref{lem:RegularValueEtc}. 

  Since the holomorphic Euler characteristic satisfies the law of conservation 
  of number \cite{GiraldoGomezMont02}, there exist open neighborhoods 
  $0 \in U \subset X$ and $(0, 0) \in T \subset \CC \times \CC^{k-1}$ such that 
  for all $(t, v) \in T$ one has 
  \[
    \chi\left(
      R(\nu_\alpha)_*\left(
      \begin{array}{c}
        {\mathcal K}osz(\nu_\alpha^*(F^0))\\
        \otimes \\
        {\mathcal E}NC(\nu_\alpha^*(\omega^0))
      \end{array}
      \right)
    \right)_0
    = \sum_{p \in U}
    \chi\left(
      R(\nu_\alpha)_*\left(
      \begin{array}{c}
        {\mathcal K}osz(\nu_\alpha^*(F^v))\\
        \otimes \\
        {\mathcal E}NC(\nu_\alpha^*(\omega^t))
      \end{array}
      \right)
    \right)_p
  \]
  For $0 \neq t'$ sufficiently small the sum on the right hand side splits 
  into a contribution of $1$ for each Morse critical point $\{p_i\}_{i=1}^\mu$
  of $f_j - t'\cdot l_j$ 
  on the regular locus of the central fiber of $F^0$ and a potential 
  residual contribution at the origin $p = 0$.
  Now let $(t, v)$ be close to $(t', 0)$ so that $v \in \Omega$ in the sense of 
  Theorem \ref{thm:MainWorkhorse} (cf. Remark \ref{rem:RegularValueRestriction}) 
  and so that $\omega^t$ degenerates on 
  $(F)^{-1}(\{v\}) \cap X$ as in Lemma \ref{lem:Morsification}. Passing from 
  $(t', 0)$ to $(t, v)$ we see no local changes at the Morse critical points 
  $p_i$, except that they may move their position. The residual point $p = 0$, 
  on the other hand, splits up into further degeneracy points
  $\{q_j\}_{j = 1}^\lambda$ of $\omega^t$, 
  the number of which coincides with 
  \[
    \lambda = \chi\left(
      R(\nu_\alpha)_*\left(
      \begin{array}{c}
        {\mathcal K}osz(\nu_\alpha^*(F^0))\\
        \otimes \\
        {\mathcal E}NC(\nu_\alpha^*(\omega^{t'}))
      \end{array}
      \right)
    \right)_0.
  \]
  As $t' \neq 0$ while $\D f_j(0)$ vanishes on any limiting tangent space at 
  the origin, this Euler characteristic coincides with the one for 
  \[
    \tilde \omega = (\D f_1,\dots, \D f_{j-1}, t' \cdot \D l_j, \D l_{j+1}, \dots, \D l_k)
  \]
  in place of $\omega^{t'}$. To see this, recall that according to 
  Theorem \ref{thm:MainWorkhorse}, $\lambda$ as above 
  is a \textit{topological} obstruction which depends only on the homotopy class 
  of the frame given by $\nu^*\omega^{t'}$ along the boundary 
  \[
    \nu^{-1}(X \cap \partial B_\varepsilon \cap F^{-1}(\{0\})).
  \]
  We claim that for $\varepsilon > 0$ sufficiently small, the homotopy classes 
  of $\nu^*\omega^{t'}$ and $\nu^* \tilde \omega$ coincide. 
  To see this, choose an auxillary metric on the frame bundle for $\tilde \Omega^1$ and 
  note that
  \[
    \max_{p \in \nu^{-1}(X \cap B_\varepsilon)} 
    |\tilde \omega(p) - \omega^{t'}(p)|
    \overset{\varepsilon \to 0}{\longrightarrow}
    0
  \]
  by continuity, as their difference tends to $\nu^*(\D f_j(0)) = 0$.
  At the same time, $\tilde \omega$ is non-degenerate along the fiber $\nu^{-1}(\{0\})$ 
  and, since the latter is compact, the distance of $\tilde \omega$ from the 
  locus of degenerate collections of sections is positively bounded away from zero.
  But then over $\nu^{-1}(\partial B_\varepsilon \cap X)$ the section
  $\omega^{t'}$ must come to lay in some tubular neighborhood of 
  $\tilde \omega$ for $\varepsilon >0$ small enough.
  
  It follows that the number 
  $\lambda$ is also computed using $\tilde \omega$ which finishes the proof 
  of the first equality in Theorem \ref{thm:LeGreuelFormulaI}.

  The second equality follows in a similar way, but using (\ref{eqn:ThimbleNumber}). 
  One just applies Theorem \ref{thm:MainWorkhorse} to the functions 
  \[
    (f_1,\dots,f_{j-1}, f_j, l_{j+1}, \dots, l_k)
  \]
  and
  \[
    (f_1,\dots,f_{j-1}, l_j, l_{j+1}, \dots, l_k);
  \]
  both with the collection of $1$-forms given by 
  \[
    (\D f_1,\dots, \D f_j, \D l_j, \dots, \D l_k).
  \]
  The proof is even easier as no eventual 'residual contribution' has to be taken 
  into account. 

  For the third and last equality, recall the strategy of proof for the classical 
  L\^e-Greuel formula. According to Lemma \ref{lem:RegularValueEtc}, we may assume that 
  the function 
  \[
    f' \colon (\CC^n, 0) \to (\CC^{k-1}, 0), \quad x \mapsto (f_1(x), \dots, f_{k-1}(x))
  \]
  has an isolated singularity on $(X,0)$ in the stratified sense. 
  An application of Theorem \ref{thm:MainWorkhorse} to this function and the collection 
  of $1$-forms given by 
  \[
    \omega = (\D f_1,\dots, \D f_{k-1}, \D f_k)
  \]
  counts the number $\lambda$ of critical points of a Morsification of $f_k$ on the 
  Milnor fiber $M_{f'}$ of $f'$ on $(X, 0)$. According to the additional assumptions 
  on the rectified homological depths for this last formula we have 
  a short exact 
  sequence\footnote{This needs to be adapted in the obvious ways for the case $d = k$.} 
  in homology with rational coefficients 
  \[
    0 \to 
    H_{d-k+1}(M_{f'}) \to 
    H_{d-k+1}(M_{f'}, M_f) \to 
    H_{d-k}(M_f) \to 
    0
  \]
  where $M_{f}$ is the Milnor fiber of the original function $f$. 
  The term in the middle is a $\QQ$-vector space of dimension $\lambda$. 
  The claim now follows by induction on the codimension $k$.
\end{proof}

\section{Computational aspects}

Advances in the field of Computer Algebra allow one to tackle increasingly 
difficult computational problems. In this last section we briefly report 
on our attempts to render the findings in this paper computable in 
OSCAR \cite{Oscar}. This section might be subject to updates in a forthcoming 
version of this preprint.

\begin{example}
  \label{exp:PreBaked2x2Surface}
  We continue with a modified version of Example \ref{exp:RunningExampleMorsification}, 
  where instead of the symmetric matrices we consider the full space
  \[
    X = \left\{ \det \begin{pmatrix} x & w \\ z & y \end{pmatrix} = 0 \right\} \subset \CC^{2 \times 2}
  \]
  and only one function 
  \[
    f \colon \CC^{2\times 2} \to \CC, \quad 
    \begin{pmatrix} x & w \\ z & y \end{pmatrix} \mapsto y - x^k.
  \]
  The same reasoning as in Example \ref{exp:RunningExampleIntroduction} applies 
  and we find that $M_{f|(X, 0)} \cong \bigvee_{i=1}^k S^2$ to be 
  a bouquet of $k$ spheres, but of dimension $2$, accounting for the higher-dimensional setup.
  Again, according to Theorem \ref{thm:BouquetDecompositionHypersurface}, one 
  of these spheres stems from the complex link $\mathcal L(X, \{0\})$ so that 
  we expect $\mu(1; f) = k-1$.
  Let us verify this using Theorem \ref{thm:LeGreuelFormulaI}
  and Corollary \ref{cor:bakingComplexes}. 

  Indeed, due to the simplicity of the example, the sum in the first formula 
  of Theorem \ref{thm:LeGreuelFormulaI} reduces to a single pair 
  \[
    \mu(1; f) = R(\nu_1)_* \left(\mathcal ENC(\nu_1^*\D f)\right)
     - R(\nu_1)_* \left(\mathcal ENC(\nu_1^*\D l)\right)
  \]
  for some sufficiently general linear form $l$ leading to a Morsification of $f|(X, 0)$. 
  Again, we will choose this linear form to be $l = x + w$.

  The complex $\mathrm{Prep}(X; 0, 1)$ from Corollary \ref{cor:bakingComplexes} takes the form 
  \begin{equation}
    \label{eqn:PreBakedComplex2x2Surface}
    \begin{xy}
      \xymatrix{
        0 & 
        R^1 \ar[l] & 
        R^9 \ar[l]_{A_{-1}} & 
        R^{16} \ar[l]_{A_{-2}} & 
        R^9 \ar[l]_{A_{-3}} & 
        R^1 \ar[l]_{A_{-4}} & 
        0 \ar[l]
      }
    \end{xy}
  \end{equation}
  in cohomological degrees $0$ to $-5$ 
  over the ring $R = \OO_{\CC^{2 \times 2}, 0}[a_{1, 1}, a_{1, 2}, a_{2, 1}, a_{2, 2}]$ 
  with matrices 
\[
  A_{-1} = 
  \begin{pmatrix}
                                -x\cdot z + y\cdot w\\
                  -a_{1, 2}\cdot y - a_{2, 2}\cdot z\\
                  -a_{1, 2}\cdot x - a_{2, 2}\cdot w\\
    -a_{1, 1}\cdot a_{2, 2} + a_{1, 2}\cdot a_{2, 1}\\
                   a_{1, 1}\cdot x + a_{1, 2}\cdot y\\
                   a_{2, 1}\cdot x + a_{2, 2}\cdot y\\
                   a_{1, 1}\cdot w + a_{1, 2}\cdot z\\
                   a_{1, 1}\cdot y + a_{2, 1}\cdot z\\
                   a_{1, 1}\cdot x + a_{2, 1}\cdot w
  \end{pmatrix}
  \quad
  A_{-2} = 
  \begin{pmatrix}
     a_{1, 2} &         w &        -z & \cdots &         0 &         0 &         0\\
     a_{2, 2} &        -x &         y & \cdots &         0 &         0 &         0\\
            0 & -a_{1, 1} &         0 & \cdots &         0 & -a_{1, 2} &         0\\
       \vdots & \vdots & \vdots & \ddots & \vdots & \vdots & \vdots \\
            0 &         0 &         0 & \cdots &         0 &         0 &         0\\
     a_{1, 1} &         0 &         0 & \cdots &         0 &        -w &         z\\
     a_{2, 1} &         0 &         0 & \cdots &         0 &         x &        -y
  \end{pmatrix}
\]

\[
  A_{-3} = 
  \begin{pmatrix}
    -a_{1, 1} &         0 & -w &        0 & \cdots &        0 & a_{1, 2} &        0\\
    -a_{2, 1} & -a_{1, 1} &  x &       -w & \cdots &        0 & a_{2, 2} & a_{1, 2}\\
            0 & -a_{2, 1} &  0 &        x & \cdots &        0 &        0 & a_{2, 2}\\
       \vdots & \vdots & \vdots & \vdots & \ddots & \vdots & \vdots & \vdots \\
            0 &         0 &  0 & a_{1, 2} & \cdots & a_{1, 2} &        0 &        0\\
            y &         z &  0 &        0 & \cdots &        0 &        x &        w
  \end{pmatrix}
\]

\[
  A_{-4} = 
  \begin{pmatrix}
    -a_{2, 1}\cdot x - a_{2, 2}\cdot y & -a_{2, 1}\cdot w - a_{2, 2}\cdot z & \cdots & -x\cdot z + y\cdot w & -a_{1, 1}\cdot a_{2, 2} + a_{1, 2}\cdot a_{2, 1}
  \end{pmatrix}
\]
We substitute the $a_{i, j}$ for the differential of $f$
\[
  \begin{pmatrix}
    a_{1, 1} & a_{1, 2} &
    a_{2, 1} & a_{2, 2} 
  \end{pmatrix}
  \mapsto 
  \begin{pmatrix}
    -k \cdot x^{k-1} & 0 &
    0 & 1
  \end{pmatrix}
\]
and obtain a new complex over the ring $\CC\{x, y, z, w\}$ 
which turns out to be homotopy equivalent to the complex
\[
  \mathrm{Prep}(X; 0, 1)|_{\D f} \cong_{\mathrm{ht}} \Kosz(k\cdot x^{k-1}, y, z, w).
\]
This complex is exact but in cohomological degree zero, where 
its homology is of length $\lambda = k-1$ as desired.

Repeating this process with $\D l$ in place of $\D f$ for the 
second summand, we find a complex which is homotopic to the 
zero complex. Therefore, the second summand for $\mu(1; f)$ 
above is equal to zero and we find 
\[
  \mu(1; f) = \lambda - 0 = k-1
\]
as anticipated.
\end{example}

\begin{example}
  \label{exp:PreBaked2x2Curve}
  We pick up on Example \ref{exp:RunningExampleCompleteIntersection} for the function
  \[
    g \colon (\CC^{2\times 2}, 0) \to (\CC^2, 0), \quad 
    \begin{pmatrix}
      x & w \\ z & y
    \end{pmatrix}
    \mapsto 
    \left(z - w, y - x^k\right)
  \]
  on $X = \{ \varphi \in \CC^{2\times 2} : \det \varphi = 0\}$. 
  The generic complex $\mathrm{Prep}(X; 1, 2)$ from 
  Corollary \ref{cor:bakingComplexes} in question for the computation of $\mu(1; g)$
  has several parameters: A scalar $c$ for the Koszul-, and a matrix of parameters 
  \[
    \begin{pmatrix}
      a_{1, 1} & a_{1, 2} & a_{2, 1} & a_{2, 2} \\
      b_{1, 1} & b_{1, 2} & b_{2, 1} & b_{2, 2} \\
    \end{pmatrix}
  \]
  for the Eagon-Northcott part. Let $R$ be the corresponding polynomial 
  ring over $\OO_{\CC^{2\times 2}, 0}$.
  Simplifying the raw output from the {\v C}ech cohomology computation up 
  to homotopy yields a complex
  \begin{equation}
    \label{eqn:PreBakedComplex2x2Surface}
    \begin{xy}
      \xymatrix{
        0 & 
        R^2 \ar[l] & 
        R^{11} \ar[l]_{A_{-1}} & 
        R^{18} \ar[l]_{A_{-2}} & 
        R^{11} \ar[l]_{A_{-3}} & 
        R^2 \ar[l]_{A_{-4}} & 
        0, \ar[l]
      }
    \end{xy}
  \end{equation}
  again in cohomological degrees $0$ to $-5$.
  The matrices read
  \[
    A_{-1} = 
    \begin{pmatrix}
      c &  0\\
      -x\cdot z + y\cdot w &  0\\
      \begin{matrix}
        a_{1, 1}\cdot b_{2, 1}\cdot x + a_{1, 1}\cdot b_{2, 2}\cdot y + a_{1, 2}\cdot b_{2, 1}\cdot y \\
        \qquad - a_{2, 1}\cdot b_{1, 1}\cdot x - a_{2, 1}\cdot b_{1, 2}\cdot y - a_{2, 2}\cdot b_{1, 1}\cdot y
      \end{matrix} & -y\\
      a_{1, 2}\cdot b_{2, 2}\cdot y - a_{2, 2}\cdot b_{1, 2}\cdot y &  x\\
      \begin{matrix}
        a_{1, 1}\cdot b_{2, 1}\cdot w + a_{1, 1}\cdot b_{2, 2}\cdot z + a_{1, 2}\cdot b_{2, 1}\cdot z \\
        \qquad - a_{2, 1}\cdot b_{1, 1}\cdot w - a_{2, 1}\cdot b_{1, 2}\cdot z - a_{2, 2}\cdot b_{1, 1}\cdot z 
      \end{matrix} & -z\\
      a_{1, 2}\cdot b_{2, 2}\cdot z - a_{2, 2}\cdot b_{1, 2}\cdot z &  w\\
      a_{1, 2}\cdot b_{2, 1}\cdot y - a_{2, 1}\cdot b_{1, 2}\cdot y - a_{2, 1}\cdot b_{2, 2}\cdot z + a_{2, 2}\cdot b_{2, 1}\cdot z & -y\\
      -a_{1, 1}\cdot b_{1, 2}\cdot y - a_{1, 1}\cdot b_{2, 2}\cdot z + a_{1, 2}\cdot b_{1, 1}\cdot y + a_{2, 2}\cdot b_{1, 1}\cdot z &  z\\
      a_{1, 2}\cdot b_{2, 1}\cdot x - a_{2, 1}\cdot b_{1, 2}\cdot x - a_{2, 1}\cdot b_{2, 2}\cdot w + a_{2, 2}\cdot b_{2, 1}\cdot w & -x\\
      -a_{1, 1}\cdot b_{1, 2}\cdot x - a_{1, 1}\cdot b_{2, 2}\cdot w + a_{1, 2}\cdot b_{1, 1}\cdot x + a_{2, 2}\cdot b_{1, 1}\cdot w &  w\\
      0 &  c\\
    \end{pmatrix}
  \]

  \[
    A_{-2} = 
    \begin{pmatrix}
      x\cdot z - y\cdot w & c & \dots & 
      0 &  0\\
      0 & 
      \begin{matrix}
        -a_{1, 1}\cdot b_{2, 2} - a_{1, 2}\cdot b_{2, 1} \\
        \qquad + a_{2, 1}\cdot b_{1, 2} + a_{2, 2}\cdot b_{1, 1} 
      \end{matrix} &
      \dots & 
      0 &  0\\
      \vdots & \vdots & \ddots & \vdots & \vdots \\
      \begin{matrix} 
        -a_{1, 2}\cdot b_{2, 1}\cdot x + a_{2, 1}\cdot b_{1, 2}\cdot x \\
        \qquad + a_{2, 1}\cdot b_{2, 2}\cdot w - a_{2, 2}\cdot b_{2, 1}\cdot w 
      \end{matrix}& 
      0 &
      \dots &
      0 &  x\\
      \begin{matrix}
        a_{1, 1}\cdot b_{1, 2}\cdot x + a_{1, 1}\cdot b_{2, 2}\cdot w \\
        \qquad - a_{1, 2}\cdot b_{1, 1}\cdot x - a_{2, 2}\cdot b_{1, 1}\cdot w 
      \end{matrix} &
      0 &
      \dots & 
      c & -w
    \end{pmatrix}
  \]

  \[
    A_{-3} = 
    \begin{pmatrix}
      \begin{matrix}
        -a_{1, 1}\cdot b_{2, 2} - a_{1, 2}\cdot b_{2, 1} \\
        \qquad + a_{2, 1}\cdot b_{1, 2} + a_{2, 2}\cdot b_{1, 1} 
      \end{matrix} 
      &                      -c & \dots &
      0 &         0\\
      -a_{1, 1}\cdot b_{2, 1} + a_{2, 1}\cdot b_{1, 1} &                       0 &
      \dots & 
      0 &         0\\
      a_{1, 2}\cdot b_{2, 2} - a_{2, 2}\cdot b_{1, 2} &                       0 &
      \dots &
      0 &         0\\
      \vdots & \vdots & \ddots & \vdots & \vdots \\
      a_{1, 2}\cdot b_{2, 1} - a_{2, 1}\cdot b_{1, 2} &                       0 & 
      \dots & 
      z &         0\\
      0 & b_{1, 2}\cdot y + b_{2, 2}\cdot z & 
      \dots &
      0 &         0\\
      0 & a_{1, 2}\cdot y + a_{2, 2}\cdot z & 
      \dots & 
      0 &         0\\
    \end{pmatrix}
  \]

  \[
    A_{-4} = 
    \begin{pmatrix}
      b_{1, 2}\cdot y + b_{2, 2}\cdot z & b_{1, 2}\cdot x + b_{2, 2}\cdot w & 
      \dots & 
      c & 0\\
      a_{1, 2}\cdot y + a_{2, 2}\cdot z & a_{1, 2}\cdot x + a_{2, 2}\cdot w &
      \dots &
      0 & c\\
    \end{pmatrix}
  \]
  We carry out the substitutions for $k=5$. 
  The first formula from Theorem \ref{thm:LeGreuelFormulaI}
  has two pairs of summands. For $j=1$ we obtain Euler 
  characteristics $6 - 2$, both concentrated in cohomological degree zero.
  Note that in this case there  
  is indeed a non-trivial correction term provided by the second 
  summand. For $j=2$, the two terms cancel, since the first 
  component $g_1 = z - w$ was already a sufficiently general 
  linear form. Thus $\mu(1; f) = 4$ as expected.
\end{example}

Going beyond these rather simple examples treated above with 
computing the full complex $\mathrm{Prep}(X; k, m)$ seems to be 
difficult. 

\begin{example}
  \label{exp:2x3Example}
  Consider the space $\CC^{2\times 3}$ and the subvariety 
  \[
    X = \left\{\varphi = 
    \begin{pmatrix}
      x & y & z \\
      u & v & w
    \end{pmatrix}
    \in \CC^{2\times 3} : \rank \varphi < 2\right\},
  \]
  as usual with its rank stratification. 
  The function 
  \[
    f \colon \CC^{2\times 3} \to \CC^2, \quad
    \begin{pmatrix}
      x & y & z \\
      u & v & w
    \end{pmatrix}
    \mapsto 
    (x-v^k, w - y^l)
  \]
  has a singularity which is isomorphic to the third entry 
  in the list of simple Cohen-Macaulay codimension $2$ surface 
  singularities, \cite[Theorem 3.3]{FruehbisKruegerNeumer10}. 
  These singularities coincide with the ``Rational Triple Points'' 
  classified by Tjurina \cite{Tjurina68} and they admit a 
  ``resolution in family'' (see e.g. \cite{FruehbisKruegerZach23}).
  The original results by Tjurina therefore reveal that 
  \[
    M_{f|(X, 0)} \cong_{\mathrm{ht}} \bigvee_{j=1}^{k + l -1} S^2.
  \]
  Since according to e.g. \cite{Zach18BouquetDecomp} we 
  have $\mathcal L^1(X, \{0\}) \cong S^2$, a comparison with 
  the Bouquet Theorem \ref{thm:BouquetDecompositionsCompleteIntersections}
  yields
  \[
    \mu(1; f) = k + l - 2.
  \]
  Computations with OSCAR for the complex 
  \[
    R(\nu_1)_* \left(\mathcal Kosz(\nu_1^*(f_1)) \otimes \mathcal ENC(\nu_1^*\D f)\right)
  \]
  with the naive approach from \cite{Zach22} using {\v C}ech cohomology 
  produce a complex of free $\CC[x, y, z, u, v, w]$-modules with 
  the ranks listed in Table \ref{tab:BettiNumbersPushforward}.
  \begin{table}
    \label{tab:BettiNumbersPushforward}
    \begin{center}
      \begin{tabular}{|r|r|r|r|r|r|r|r|}
        \hline
        cohom. degree & $-9$ & $-8$ & $-7$ & $-6$ & $-5$ & $-4$ \\
        \hline
        rank & 0 & 18 & 3627 & 51282 & 308043 & 1063948\\
        \hline
        \hline
        cohom. degree & $-2$ & $-1$ & $0$ & $1$ & $2$ & $-3$ \\
        \hline
        rank & 3905782 & 4734600 & 4420820 & 3155041& 1639720  & 2416310 \\
        \hline
        \hline
        cohom. degree & $3$ & $4$ & $5$ & $6$ & & \\
        \hline
        rank &  565580 &  109756 &    8125 &       0 & & \\
        \hline
      \end{tabular}
    \end{center}
    \caption{Ranks in the raw derived pushforward complex from 
    Example \ref{exp:2x3Example} using {\v C}ech cohomology}
  \end{table}
  This complex can be simplified massively to a smaller one, homotopy-equivalent to the original;
  but the simplification itself takes too long to be of any practical use 
  for the moment. 
\end{example}

\begin{remark}
  \label{rem:ComputationalImprovements}
  It is not unlikely that the full computation of the complexes $\mathrm{Prep}(X; k, m)$ 
  might be impractical in general. However, we have not yet tried to achieve this 
  with e.g. the algorithm using the exterior algebra from \cite{EisenbudSchreyer08}.
  
  Should that turn out to also not be fruitful, one still has the possibility to only 
  compute the spectral sequence of the double complex induced on the direct image 
  of the individual sheaves appearing in 
  $\mathcal Kosz(\nu^*f) \otimes \mathcal ENC(\nu^*\omega)$ from Theorem 
  \ref{thm:MainWorkhorse}. We plan to explore these possibilities further in the future. 
\end{remark}

\section*{Acknowledgements}
I would like to thank Xavier G\'omez-Mont for inspiring discussions during a 
visit in Guanajuato this year. This extends to the organizers of the 
``International Congress on Complex Geometry, Singularities and Dynamics'' 
for their generous invitation to Cuernavaca.
This work has been funded by the Deutsche Forschungsgemeinschaft 
(DFG, German Research Foundation) – Project-ID 286237555 – TRR 195.

\printbibliography
\end{document}